\documentclass[reqno,11pt]{amsart}

\usepackage{fullpage}

\usepackage{amssymb}
\usepackage{dsfont}
\usepackage[all]{xy}

\usepackage[utf8]{inputenc} % Required for inputting international characters
\usepackage[T1]{fontenc} % Output font encoding for international characters

\usepackage{hyperref}

\usepackage[nobysame,alphabetic,initials,msc-links]{amsrefs}

\DefineSimpleKey{bib}{how}

\renewcommand{\eprint}[1]{#1}
\BibSpec{misc}{%
  +{}{\PrintAuthors}  {author}
  +{,}{ \textit}      {title}
  +{,}{ }             {how}
  +{}{ \parenthesize} {date}
  +{,} { available at \eprint}        {eprint}
  +{,}{ available at \url}{url}
  +{,}{ }             {note}
  +{.}{}              {transition}
}

\numberwithin{equation}{section}

\theoremstyle{plain}%default
\newtheorem{theorem}{Theorem}[section]
\newtheorem{proposition}[theorem]{Proposition}
\newtheorem{lemma}[theorem]{Lemma}
\newtheorem{corollary}[theorem]{Corollary}

\theoremstyle{definition}
\newtheorem{definition}[theorem]{Definition}

\theoremstyle{remark}
\newtheorem{remark}[theorem]{Remark}
\newtheorem{example}[theorem]{Example}

\newtheorem{condition}[theorem]{Condition}

\newcommand\bp{\begin{proof}}
\newcommand\ep{\end{proof}}
\newcommand\er{\nopagebreak\mbox{\ }\hfill$\diamondsuit$}

\newcommand{\un}{\mathds{1}}

\newcommand\C{\mathbb{C}}

\newcommand\Z{\mathbb{Z}}

\newcommand{\G}{\mathcal{G}}
\newcommand{\F}{\mathcal{F}}
\newcommand{\J}{\mathcal{J}}

\newcommand{\CcG}{{C_{c}(\mathcal{G})}}

\newcommand{\Gu}{{\mathcal{G}^{(0)}}}
\newcommand{\Gxx}{\mathcal{G}^x_x}

\newcommand\dom{\operatorname{dom}}
\newcommand\id{\operatorname{id}}

\newcommand\ord{\operatorname{ord}}
\newcommand\supp{\operatorname{supp}}

\newcommand\ess{\mathrm{ess}}
\newcommand\sing{\mathrm{sing}}

\newcommand\eps{\varepsilon}

\begin{document}

\title{Non-Hausdorff \'etale groupoids and C$^*$-algebras of left cancellative monoids}

\date{January 26, 2022; revised January 9, 2023}

\author{Sergey Neshveyev}
\address{University of Oslo, Mathematics institute}
\email{sergeyn@math.uio.no}
\thanks{S.N. is supported by the NFR project 300837 ``Quantum Symmetry''.}

\author{Gaute Schwartz}
\email{gautesc@math.uio.no}

\begin{abstract}
We study the question whether the representations defined by a dense subset of the unit space of a locally compact \'etale groupoid are enough to determine the reduced norm on the groupoid C$^*$-algebra. We present sufficient conditions for either conclusion, giving a complete answer when the isotropy groups are torsion-free. As an application we consider the groupoid $\G(S)$ associated to a left cancellative monoid $S$ by Spielberg and formulate a sufficient condition, which we call C$^*$-regularity, for the canonical map $C^*_r(\G(S))\to C^*_r(S)$ to be an isomorphism, in which case~$S$ has a well-defined full semigroup C$^*$-algebra $C^*(S)=C^*(\G(S))$. We give two related examples of left cancellative monoids $S$ and $T$ such that both are not finitely aligned and have non-Hausdorff associated \'etale groupoids, but $S$ is C$^*$-regular, while $T$ is not.
\end{abstract}

\maketitle

\section*{Introduction}

The C$^*$-algebras of non-Hausdorff locally compact groupoids were introduced by Connes in \cite{MR679730}, where the main examples were given by the holonomy groupoids of foliations. It is known that some of the basic properties of groupoid C$^*$-algebras of Hausdorff groupoids can fail in the non-Hausdorff case. One of such properties is that to compute the reduced norm it suffices to consider the representations $\rho_x\colon C_c(\G)\to B(L^2(\G_x))$ for $x$ running through any dense subset $Y\subset\Gu$. A simple counterexample is provided by the line with a double point. The first systematic study of which extra conditions on $Y$ one needs was carried out by Khoshkam and Skandalis~\cite{KhSk}. Our starting point is the simple observation, which can be viewed as a reformulation of a result in~\cite{KhSk}, that for \'etale groupoids it suffices to require that for every point $x\in\Gu\setminus Y$ there is a net in~$Y$ converging to $x$ and having no other accumulation points in $\Gxx$. As we show, this condition is in general not necessary, but it becomes so if the isotropy groups $\Gxx$ do not have too many finite subgroups, in particular, if they are torsion-free.

Our motivation for studying these questions comes from the problem of defining a full semigroup C$^*$-algebra of a left cancellative monoid. Every such monoid $S$ has a regular representation on~$\ell^2(S)$ and hence a well-defined reduced C$^*$-algebra $C^*_r(S)$. It is natural to try to define the full semigroup C$^*$-algebra as a universal C$^*$-algebra generated by isometries $v_s$, $s\in S$, such that $v_sv_t=v_{st}$, but one quickly sees that more relations are needed to get an algebra that is not unreasonably bigger than~$C^*_r(S)$. A major progress in this old problem was made by Li~\cite{MR2900468}, who realized that in~$C^*_r(S)$ there are extra relations coming from the action of $S$ on the constructible ideals of $S$, which are right ideals of the form $s_1^{-1}t_1\dots s_n^{-1}t_nS$. Soon afterwards Norling~\cite{MR3200323} observed that this has an interpretation in terms of the left inverse hull $I_\ell(S)$ of $S$: the C$^*$-algebra $C^*_r(S)$ is obtained by reducing the reduced C$^*$-algebra of the inverse semigroup $I_\ell(S)$ to an invariant subspace of its regular representation, and so the new relations in $C^*_r(S)$ arise from those in $C^*_r(I_\ell(S))$. Since the representations of inverse semigroups are a well-studied subject and the corresponding C$^*$-algebras have groupoid models defined by Paterson~\cite{MR1724106}, this opened the possibility to defining $C^*(S)$ as a groupoid C$^*$-algebra.

Specifically (see Section~\ref{sec:monoid} for details), the subrepresentation of the regular representation of~$I_\ell(S)$ defining $C^*_r(S)$ gives rise to a reduction $\G_P(S)$ of the Paterson groupoid of $I_\ell(S)$ and to a surjective homomorphism $C^*_r(\G_P(S))\to C^*_r(S)$. When this map is an isomorphism, it is natural to define~$C^*(S)$ as $C^*(\G_P(S))$. This C$^*$-algebra can be described in terms of generators and relations, since there is such a description for $C^*(I_\ell(S))$, and simultaneously its definition as a groupoid C$^*$-algebra subsumes a number of results on (partial) crossed product decompositions of semigroup C$^*$-algebras. The trouble, however, is that this does not work for all $S$, the map $C^*_r(\G_P(S))\to C^*_r(S)$ is not always an isomorphism.

In~\cite{MR4151331} Spielberg introduced, in a more general context of left cancellative small categories, a quotient $\G(S)$ of $\G_P(S)$ that kills some ``obvious'' elements in the kernel of $C^*_r(\G_P(S))\to C^*_r(S)$ (see Proposition~\ref{prop:G_P=G}). But as he showed, the canonical homomorphism $C^*_r(\G(S))\to C^*_r(S)$ can still have a nontrivial kernel. It should be said that the kernel of $C^*_r(\G(S))\to C^*_r(S)$ is small: under rather general assumptions (for example, for all countable $S$ with trivial group of units) $C^*_r(S)$ can be identified with the essential groupoid C$^*$-algebra $C^*_\ess(\G(S))$ of $\G(S)$, as defined by Kwa\'{s}niewski and Meyer~\cite{MR4246403}. Still, we do not think that this is enough to call $C^*(\G(S))$ the full semigroup C$^*$-algebra of $S$ when $C^*_r(\G(S))\ne C^*_r(S)$.

Spielberg showed that there are two sufficient conditions for the equality $C^*_r(\G(S))=C^*_r(S)$, one is that $\G(S)$ is Hausdorff, the other is that $S$ is finitely aligned, which is equivalent to saying that every constructible ideal of $S$ is finitely generated. Already the first condition covers, for example, all group embeddable monoids. For such monoids we have $\G(S)=\G_P(S)$, and the corresponding full semigroup C$^*$-algebras $C^*(S)=C^*(\G(S))=C^*(\G_P(S))$ have been recently comprehensively studied by Laca and Sehnem~\cite{LS}.

For general $S$, the question whether $C^*_r(\G(S))\to C^*_r(S)$ is an isomorphism is exactly the type of question we started with: can the reduced norm on $C_c(\G(S))$ be computed using certain dense subset $Y=\{\chi_s|s\in S\}$ of $\G(S)^{(0)}$? In this formulation it is immediate that the answer is ``yes'' when $\G(S)$ is Hausdorff. In the non-Hausdorff case we can try to use our general results to arrive to either conclusion. This leads to a simple (to formulate, but in general not to check) sufficient condition for the equality $C^*_r(\G(S))=C^*_r(S)$ that we call C$^*$-regularity. We give an example of a C$^*$-regular monoid $S$ that is not finitely aligned and such that the groupoid $\G(S)$ is non-Hausdorff. A small modification of $S$ gives a monoid $T$ with $C^*_r(\G(T))\ne C^*_r(T)$. It is interesting that the kernel of $C^*_r(\G(T))\to C^*_r(T)$ has nonzero elements already in the $*$-algebra generated by the canonical elements $v_t$, $t\in T$, so in some sense $\G(T)$ is a wrong groupoid model for the semigroup $*$-algebra of~$T$ already at the purely algebraic level.

Let us finally mention that an interesting related problem is to find groupoid models for boundary quotients of semigroup C$^*$-algebras, but we are not going to touch it in the present paper.

\bigskip

\section{\texorpdfstring{C$^*$}{C*}-algebras of non-Hausdorff \'etale groupoids}\label{sec:groupoid}

Assume $\G$ is a locally compact, not necessarily Hausdorff, \'{e}tale groupoid. By this we mean that~$\G$ is a groupoid endowed with a locally compact topology such that
\begin{enumerate}
\item[-] the groupoid operations are continuous;
\item[-] the unit space $\Gu $ is a locally compact Hausdorff space in the relative topology;
\item[-] the range map $r\colon\G\to\Gu$ and the source map $s\colon\G\to\Gu$ are local homeomorphisms.
\end{enumerate}

For an open Hausdorff subset $V\subset\G$, consider the usual space $C_c(V)$ of continuous compactly supported functions on $V$. Every such function can be extended by zero to $\G$; in general this extension is not a continuous function on $\G$. This way we can view $C_c(V)$ as a subspace of the space of functions $\operatorname{Func}(\G)$ on $\G$. For arbitrary open subsets $U\subset\G$ we denote by $C_c(U)\subset\operatorname{Func}(\G)$ the linear span of the subspaces $C_c(V)\subset\operatorname{Func}(\G)$ for all open Hausdorff subsets $V\subset U$. Instead of all possible $V$ it suffices to take a collection of open bisections covering $U$.

The space $C_c(\G)$ is a $*$-algebra with the convolution product
\begin{equation*} \label{eprod}
(f_{1}*f_{2})(g) := \sum_{h \in \G^{r(g)}} f_{1}(h) f_{2}(h^{-1}g)\quad \text{for}\quad g\in \G,
\end{equation*}
and involution $f^*(g)=\overline{f(g^{-1})}$, where $\G^x=r^{-1}(x)$. The full groupoid C$^*$-algebra $C^*(\G)$ is defined as the C$^*$-enveloping algebra of $C_c(\G)$.

For every $x\in\Gu$, define a representation $\rho_{x}\colon C_c(\G) \to B(\ell^{2}(\G_{x}))$, where $\G_x=s^{-1}(x)$, by
\begin{equation*}\label{eq:rhox}
(\rho_{x}(f)\xi)(g) =\sum_{h \in \G^{r(g)}} f(h) \xi(h^{-1}g).
\end{equation*}
Then the reduced C$^*$-algebra $C^*_r(\G)$ is defined as the completion of $C_c(\G)$ with respect to the norm
\begin{equation*}
\lVert f \rVert_{r}  = \sup_{x\in \Gu } \lVert \rho_{x}(f) \rVert.
\end{equation*}

Recall (see, e.g., \cite{MR2419901}*{Section~3}) that for all $f\in \CcG$ we have the inequalities $\lVert f\rVert_{\infty} \leq \lVert f\rVert_{r} \leq  \lVert f\rVert$, where $\lVert \cdot \rVert_{\infty}$ denotes the supremum-norm, and if $f\in C_c(U)$ for an open bisection $U$, then
\begin{equation*} \label{eq:norm-bisection}
\|f\|=\lVert f\rVert_r=\lVert f\rVert_{\infty}.
\end{equation*}

For a closed (in $\Gu$) invariant subset $X\subset\Gu$, denote by $\G_X$ the subgroupoid $r^{-1}(X)=s^{-1}(X)\subset\G$. In the second countable case the next result and the subsequent corollary follow easily from Renault's disintegration theorem, cf.~\cite{MR1191252}*{Remark~4.10}. The case of \'etale groupoids allows for the following elementary proof without any extra assumptions on $\G$.

\begin{proposition}\label{prop:exact-sequence-groupoid}
Assume $\G$ is a locally compact \'etale groupoid and $X\subset\G^{(0)}$ is a closed invariant subset. Then the following sets coincide:
\begin{enumerate}
\item[(1)] the kernel of the $*$-homomorphism $C^*(\G)\to C^*(\G_X)$, $C_c(\G)\ni f\mapsto f|_{\G_X}$;
\item[(2)] the closure of $C_c(\G\setminus\G_X)$ in $C^*(\G)$;
\item[(3)] the closed ideal of $C^*(\G)$ generated by $C_0(\Gu\setminus X)\subset C_0(\Gu)$.
\end{enumerate}
\end{proposition}

\bp
The sets in (2) and (3) coincide, since $C_c(\G\setminus\G_X)$ is an ideal in $C_c(\G)$ (with respect to the convolution product) and for every $f\in C_c(G\setminus\G_X)$ we can find $f'\in C_c(\Gu\setminus X)$ such that $f*f'=f$. It is also clear that $C_c(\G\setminus\G_X)$ is contained  in the kernel of the $*$-homomorphism $C^*(\G)\to C^*(\G_X)$. It follows that in order to prove the proposition it suffices to show that every representation of~$C_c(\G)$ on a Hilbert space that vanishes on $C_c(\G\setminus\G_X)$ factors through $C_c(\G_X)$. For this, in turn, it suffices to prove that $C_c(\G\setminus\G_X)$ is dense, with respect to the norm on $C^*(\G)$, in the space of functions $f\in C_c(\G)$ such that $f|_{\G_X}=0$.

\smallskip

Let us first prove the following claim. Assume $f=\sum^n_{i=1}f_i\in C_c(\G)$ satisfies $\|f|_{\G_X}\|_\infty<\eps$ for some $\eps>0$, $f_i\in C_c(U_i)$ and open bisections $U_i$. Then there exist functions $\tilde f_i\in C_c(U_i)$ such that
$$
f-\sum^n_{i=1}\tilde f_i\in C_c(\G\setminus\G_X)\quad\text{and}\quad\|\tilde f_i\|_\infty<2^n\eps\quad\text{for}\quad i=1,\dots,n.
$$

The proof is by induction on $n$. As the base of induction we take $n=0$, meaning that $f=0$. In this case there is nothing to prove. So assume the claim is true for some $n\ge0$. For the induction step assume $f\in C_c(\G)$ satisfies $\|f|_{\G_X}\|_\infty<\eps$ and we can write $f=\sum^{n+1}_{i=1}f_i$ for some $f_i\in C_c(U_i)$ and open bisections $U_i$. Let $K_{n+1}\subset U_{n+1}$ be the support of $f_{n+1}|_{U_{n+1}}$. Consider the set
$K=K_{n+1}\setminus\bigcup^n_{i=1}U_i$. As $f=f_{n+1}$ on $K$, we have
$\|f_{n+1}|_{K\cap\G_X}\|_\infty<\eps$. Hence there exists an open neighbourhood $U$ of $K\cap\G_X$ in $U_{n+1}$ such that $\|f_{n+1}|_U\|_\infty<\eps$. Let $V$ be an open neighbourhood of $K\setminus U$ in $U_{n+1}$ such that $\bar V\cap U_{n+1}\cap\G_X=\emptyset$. Then the open sets $U_1\cap U_{n+1},\dots,U_n\cap U_{n+1},U,V$ cover $K_{n+1}$. Hence we can find functions $\rho_1,\dots,\rho_n,\rho_U,\rho_V\in C_c(U_{n+1})$ taking values in the interval $[0,1]$ such that
$\supp\rho_i\subset U_i\cap U_{n+1}$, $\supp\rho_U\subset U$, $\supp\rho_V\subset V$ and
$$
\sum^n_{i=1}\rho_i(g)+\rho_U(g)+\rho_V(g)=1\quad\text{for all}\quad g\in K_{n+1}.
$$

Define $f'_i=f_i+\rho_if_{n+1}$ (pointwise product) for $i=1,\dots, n$, $f'=\sum^n_{i=1}f_i'$ and $\tilde f_{n+1}=\rho_Uf_{n+1}$. Then $f_i'\in C_c(U_i)$, $\tilde f_{n+1}\in C_c(U_{n+1})$ and we have
$$
f-f'-\tilde f_{n+1}=\rho_Vf_{n+1}\in C_c(\G\setminus\G_X).
$$
We also have $\|\tilde f_{n+1}\|_\infty\le\|f_{n+1}|_U\|_\infty<\eps< 2^{n+1}\eps$. It follows that
$$
\|f'|_{\G_X}\|_\infty=\|(f-\tilde f_{n+1})|_{\G_X}\|_\infty\le\|f|_{\G_X}\|_\infty+\|\tilde f_{n+1}|_{\G_X}\|_\infty<2\eps.
$$
We can therefore apply the inductive hypothesis to $f'$ and $2\eps$ and find functions $\tilde f_i\in C_c(U_i)$, $i=1,\dots,n$, such that
$$
f'-\sum^n_{i=1}\tilde f_i\in C_c(\G\setminus\G_X)\quad\text{and}\quad\|\tilde f_i\|_\infty<2^{n}2\eps=2^{n+1}\eps\quad\text{for}\quad i=1,\dots,n.
$$
Then the functions $\tilde f_1,\dots,\tilde f_{n+1}$ have the required properties.

\smallskip

Now, if $f\in C_c(\G)$ satisfies $f|_{\G_X}=0$, we write $f=\sum^n_{i=1}f_i$ for some $f_i\in C_c(U_i)$ and open bisections~$U_i$ and apply the above claim to an arbitrarily small $\eps>0$. Recalling that the norms $\lVert\cdot\rVert$ and $\lVert\cdot\rVert_\infty$ coincide on $C_c(U)$ for any open bisection $U$, we conclude that there is a function $\tilde f=\sum^n_{i=1}\tilde f_i\in C_c(\G)$ such that $f-\tilde f\in C_c(\G\setminus\G_X)$ and $\|\tilde f\|\le n2^n\eps$. Hence $f$ lies in the closure of $C_c(\G\setminus\G_X)$.
\ep

\begin{corollary}
We have a short exact sequence
$$
0\to C^*(\G\setminus\G_X)\to C^*(\G)\to C^*(\G_X)\to0.
$$
\end{corollary}

\bp
Since the restriction map $C_c(\G)\to C_c(\G_X)$ is surjective, the fact that the sets in (1) and (2) coincide implies that we have an exact sequence $C^*(\G\setminus\G_X)\to C^*(\G)\to C^*(\G_X)\to0$. Therefore we only need to explain why the map $C^*(\G\setminus\G_X)\to C^*(\G)$ is injective. For this it suffices to show that any nondegenerate representation $\pi\colon C^*(\G\setminus\G_X)\to B(H)$ extends to $C^*(\G)$. Since $C_c(\G\setminus\G_X)$ is an ideal of $C_c(\G)$, we can define a representation $\tilde\pi$ of $C_c(\G)$ on $\pi(C_c(\G\setminus\G_X))H$ by possibly unbounded operators in the standard way: for $f\in C_c(\G)$, put $\tilde\pi(f)\pi(f')\xi=\pi(f*f')\xi$. On $C_c(\Gu)$ this agrees with the unique extension of $\pi|_{C_0(\Gu\setminus X)}$ to a representation of $C_0(\Gu)$. Hence $\|\tilde\pi(f)\|\le\|f\|_\infty$ for $f\in C_c(\Gu)$, and then $\|\tilde\pi(f)\|\le\|f\|$ for any open bisection $U$ and $f\in C_c(U)$, as $f^**f\in C_c(\Gu)$.
\ep

\begin{remark}[cf.~{\cite{CN}*{Remark~2.9}}]\label{rem:exotic}
Since the ideal $C_c(\G\setminus\G_X)\subset C_c(\G)$ is dense with respect to the norm on $C^*(\G)$ in the space of functions $f\in C_c(\G)$ such that $f|_{\G_X}=0$, it is also dense with respect to the reduced norm. It follows that there is a C$^*$-norm on $C_c(\G_X)$ dominating the reduced norm such that for the corresponding completion $C^*_e(\G_X)$ the sequence
$$
0\to C^*_r(\G\setminus\G_X)\to C^*_r(\G)\to C^*_e(\G_X)\to0
$$
is exact.\hfill$\diamondsuit$
\end{remark}

We now turn to the question when a set of representations $\rho_y$, $y\in Y\subset\Gu$, determines the reduced norm on $C_c(\G)$. It is easy to see that if $Y$ is $\G$-invariant, which we may always assume since the equivalence class of $\rho_x$ depends only on the orbit of $x$, a necessary condition is that $Y$ is dense in $\Gu$. But this is not enough in the non-Hausdorff case. We start with the following sufficient condition.

\begin{proposition}\label{prop:weak-containment}
Let $\G$ be a locally compact \'etale groupoid, $Y\subset\Gu$ and $x\in\Gu\setminus Y$. Assume there is a net $(y_i)_i$ in $Y$ such that $x$ is the only accumulation point of $(y_i)_i$ in~$\Gxx=\G_x\cap\G^x$. Then the representation $\rho_x$ of $C^*_r(\G)$ is weakly contained in $\bigoplus_{y\in Y}\rho_y$.
\end{proposition}

\bp
We may assume that $y_i\to x$. For every $g\in\G_x$ we then choose a net $(g_i)_i$ converging to $g$ as follows. Let $U$ be an open bisection containing $g$. Then for all $i$ large enough we have $y_i\in s(U)$, and for every such $i$ we take the unique point $g_i\in U\cap\G_{y_i}$. For all other indices $i$ we put $g_i=y_i$.

Take $g,h\in\G_x$. Observe that by our assumptions if $g\ne h$, then $g_i\ne h_i$ for all $i$ large enough, since otherwise we could first conclude that $r(g)=r(h)$ and then that $g^{-1}h\in\Gxx$ is an accumulation point of $(y_i)_i$.

Next, take an open bisection $V$ and $f\in C_c(V)$. Then
$$
(\rho_x(f)\delta_g,\delta_h)=f(hg^{-1}),\qquad (\rho_{y_i}(f)\delta_{g_i},\delta_{h_i})=f(h_ig^{-1}_i).
$$
These equalities and the observation above imply that in order to prove the proposition it suffices to show that $f(h_ig^{-1}_i)\to f(hg^{-1})$.

Assume first that $hg^{-1}\in V$. As $V$ is open and $h_ig^{-1}_i\to hg^{-1}$, it follows that eventually $h_ig^{-1}_i\in V$. But then $f(h_ig^{-1}_i)\to f(hg^{-1})$ by the continuity of $f$ on $V$.

Assume next that $hg^{-1}\notin V$. It is then enough to show that eventually $h_ig_i^{-1}$ does not lie in the support $K$ of $f|_V$. Suppose this is not the case. Then by passing to a subnet we may assume that $h_ig_i^{-1}\to w$ for some $w\in K$. Since we also have $h_ig_i^{-1}\to hg^{-1}$, we must have $r(w)=r(h)$ and $s(w)=r(g)$. Then $h^{-1}wg\in\Gxx$, $h^{-1}wg\ne x$ and $y_i=h_i^{-1}(h_ig_i^{-1})g_i\to h^{-1}wg$, which contradicts our assumptions.
\ep

\begin{remark}
The above proposition can also be deduced from results in~\cite{KhSk}*{Section~2}. In order to make the connection to~\cite{KhSk} more transparent, let us reformulate the assumptions of Proposition~\ref{prop:weak-containment} as follows. The functions $f|_{\Gu}$ for $f\in C_c(\G)$ generate a C$^*$-subalgebra $B$ of the algebra of bounded Borel functions on $\Gu$ equipped with the supremum-norm. Let $Z$ be the spectrum of $B$. As every point of $\Gu$ defines a character of $B$, we have an injective Borel map $i\colon\Gu\to Z$ with dense image. We claim that a net as in Proposition~\ref{prop:weak-containment} exists if and only if $i(x)\in\overline{i(Y)}$.

In order to show this, assume first that $(y_i)_i$ is a net in $Y$ converging to $x$ and having no other accumulation points in $\Gxx$. We claim that then $i(y_i)\to i(x)$. It suffices to show that $f(y_i)\to f(x)$ for every open bisection $U$ and $f\in C_c(U)$. If $x\in U$, this is true by continuity of $f|_U$. If $x\notin U$, then the net $(y_i)_i$ does not have any accumulation points in $U$ and therefore it eventually lies outside the support of $f|_U$, so again $f(y_i)\to f(x)$. Conversely, assume we have a net $(y_i)_i$ in $Y$ such that $i(y_i)\to i(x)$. Then obviously $y_i\to x$. Take $g\in\Gxx\setminus\{x\}$, an open bisection $U$ containing $g$ and $f\in C_c(U)$ such that $f(g)\ne0$. As $f(y_i)\to f(x)=0$ by assumption, we conclude that $g$ cannot be an accumulation point of $(y_i)_i$.
\hfill$\diamondsuit$
\end{remark}

If $x\in(\Gu\cap\bar Y)\setminus Y$, then nonexistence of a net as in Proposition~\ref{prop:weak-containment} is equivalent to the following property:

\begin{condition}\label{condition1}
For some $n\ge1$, there are elements $g_1,\dots,g_n\in\Gxx\setminus\{x\}$, open bisections $U_1,\dots,U_n$ such that $g_k\in U_k$ and a neighbourhood $U$ of $x$ in $\Gu$ satisfying $Y\cap U\subset U_1\cup\dots\cup U_n$.
\end{condition}

Indeed, if this condition is satisfied, then any net in $Y$ converging to $x$ has one of the elements $g_1,\dots,g_n$ as its accumulation point. Conversely, assume Condition~\ref{condition1} is not satisfied. For every $g\in \Gxx\setminus\{x\}$ choose an open bisection $U_g$ containing $g$. Then for every finite set $F=\{g_1,\dots,g_n\}\subset \Gxx\setminus\{x\}$ and every neighbourhood $U$ of $x$ in $\Gu$ we can find $y_{F,U}\in (Y\cap U)\setminus (U_{g_1}\cup\dots\cup U_{g_n})$. Then $(y_{F,U})_{F,U}$, with the obvious partial order defined by inclusion of $F$'s and containment of $U$'s, is the required net.

\begin{remark}
Following the terminology of~\cite{MR4246403}, a point $x\in\Gu$ is called dangerous if there is a net in $\Gu$ converging to $x$ and to a point in $\Gxx\setminus\{x\}$. Therefore the set of points $x\in(\Gu\cap\bar Y)\setminus Y$ satisfying Condition~\ref{condition1} is a subset of dangerous points. As a consequence, if $Y$ is dense in $\Gu$ and $\G$ can be covered by countably many open bisections, then by~\cite{MR4246403}*{Lemma~7.15} the set of points $x\in\Gu\setminus Y$ satisfying Condition~\ref{condition1} is meager in $\Gu$.
\hfill$\diamondsuit$
\end{remark}

If $Y$ is $\G$-invariant and Condition~\ref{condition1} is satisfied for $n=1$, then $\rho_x$ is not weakly contained in $\bigoplus_{y\in Y}\rho_y$. In order to see this, take an open neighbourhood $V\subset U$ of $x$ such that $V\subset r(U_1)\cap s(U_1)$ and a function $f\in C_c(V)$ such that $f(x)\ne0$. Then it is easy to check that $0\ne f*(\un_{U_1}-\un_U)*f\in\ker\rho_y$ for all $y\in Y$. A simple example of such a situation is the real line with a double point at $0$, cf.~\cite{KhSk}*{Example~2.5}.

But in  general, as we will see soon,  Condition~\ref{condition1} is not enough to conclude that $\rho_x$ is not weakly contained in $\bigoplus_{y\in Y}\rho_y$.
%(Thus, a dangerous point does not necessarily cause problems.)
A sufficient extra condition is given by the following proposition.

\begin{proposition}\label{prop:not-weakly-contained}
Assume $\G$ is a locally compact \'etale groupoid, $Y\subset\Gu$ is a $\G$-invariant subset and $x\in\Gu\setminus Y$ is a point satisfying Condition~\ref{condition1} such that
$$
\sum^n_{k=1}\frac{1}{\ord(g_k)}<1,
$$
where $\ord(g_k)$ is the order of $g_k$ in $\Gxx$. Then $\rho_x$ is not weakly contained in $\bigoplus_{y\in Y}\rho_y$.
\end{proposition}

For the proof we need the following simple lemma.

\begin{lemma}\label{lem:non-weakly-contained}
Let $A=C^*(a)$ be a C$^*$-algebra generated by a contraction $a$. Assume that for some $m\in\{2,3,\dots,+\infty\}$ we have a $*$-homomorphism $\pi\colon A\to C^*(\Z/m\Z)$  such that $\pi(a)=u$, where~$u$ is the unitary generator of $C^*(\Z/m\Z)$. Take numbers $\alpha>0$ and $\eps>0$ and denote by $\Omega_{\alpha,\eps}$ the convex set of states $\varphi$ on $A$ such that
$$
\varphi\ge\alpha\sum^p_{l=1}\lambda_l\chi_l
$$
for some $p\ge1$, $\lambda_1,\dots,\lambda_p\ge0$, $\sum^p_{l=1}\lambda_l=1$, and characters $\chi_1,\dots,\chi_p\colon A\to\C$ such that $|1-\chi_l(a)|<\eps$ for all $l$. Then, for every $\alpha>1/m$, there is $\eps>0$ depending only on $m$ and $\alpha$ such that $\tau\circ\pi$ does not belong to the weak$^*$ closure of $\Omega_{\alpha,\eps}$, where $\tau$ is the canonical trace on $C^*(\Z/m\Z)$.
\end{lemma}

Here we use the convention $\Z/m\Z=\Z$ for $m=+\infty$.

\bp
Assume first that $m$ is finite. Consider the positive element $b\in A$ defined by
$$
b=\frac{1}{m^2}\sum^m_{k,l=1}(a^k)^*a^l.
$$
Then $\tau(\pi(b))=1/m$. On the other hand, if $\chi$ is a character on $A$ such that $|1-\chi(a)|<\eps$, then
$$
|1-\chi(a)^k|\le k|1-\chi(a)|<m\eps\quad\text{for all}\quad 1\le k\le m,
$$
hence, assuming $m\eps<1$, we have
$$
\chi(b)=\Big|\frac{1}{m}\sum^m_{k=1}\chi(a)^k\Big|^2>(1-m\eps)^2
$$
and therefore
$$
\varphi(b)\ge\alpha (1-m\eps)^2\quad\text{for all}\quad \varphi\in\Omega_{\alpha,\eps}.
$$
It follows that $\tau\circ\pi\notin\bar\Omega_{\alpha,\eps}$ as long as $\eps$ is small enough so that $\alpha (1-m\eps)^2>1/m$.

\smallskip

Assume now that $m=+\infty$. Choose $m'\ge1$ such that $1/m'<\alpha$. Then the same arguments as above with $m$ replaced by $m'$ show that $\tau\circ\pi\notin\bar\Omega_{\alpha,\eps}$ as long as $1-m'\eps>1/\sqrt{m'\alpha}$.
\ep

\bp[Proof of Proposition~\ref{prop:not-weakly-contained}]
%We may assume that the elements $g_1,\dots,g_n$ are pairwise different.
Let $g_1,\dots,g_n$ be as in the formulation of the proposition and $U,U_1,\dots,U_n$ be given by Condition~\ref{condition1}. Choose functions $f_k\in C_c(U_k)$ such that $0\le f_k\le1$ and $f_k(g_k)=1$. Consider the C$^*$-subalgebras $A_k$ of $C^*_r(\G)$ generated by $f_k$.

Consider the restriction map $C_c(\G)\to C_c(\Gxx)$, $f\mapsto f|_{\Gxx}$. It extends to a completely positive contraction $\vartheta_{x,r}\colon C^*_r(\G)\to C^*_r(\Gxx)$, with the elements $f_k$ contained in its multiplicative domain, see~\cite{CN}*{Lemmas~1.2 and~1.4}. By restricting $\vartheta_{x,r}$ to $A_k$ we therefore get $*$-homomorphisms $\pi_k\colon A_k\to C^*_r(\Gxx)$. The image of $\pi_k$ is $C^*(G_k)\subset C^*_r(\Gxx)$, where $G_k$ is the subgroup of $\Gxx$ generated by $g_k$. Therefore if we let $m_k=\ord(g_k)$, then we can view each $\pi_k$ as a $*$-homomorphism $A_k\to C^*(\Z/m_k\Z)$.

Choose numbers $\alpha_k>1/m_k$ such that $\sum^n_{k=1}\alpha_k<1$. Let $\eps_k>0$ be given by Lemma~\ref{lem:non-weakly-contained} for the homomorphism $\pi_k\colon A_k\to C^*(\Z/m_k\Z)$, $\alpha=\alpha_k$ and $a=f_k$. Put $\eps=\min_{1\le k\le n}\eps_k$. Choose an open neighbourhood $V$ of $x$ in $\Gu$ such that $V\subset U$ and
\begin{equation}\label{eq:fk}
f_k(g)>1-\eps\quad\text{for all}\quad g\in r^{-1}(V)\cap U_k\quad\text{and}\quad 1\le k\le n.
\end{equation}
Let $f\in C_c(V)$ be such that $0\le f\le 1$ and $f(x)=1$.

Now, denote $\bigoplus_{y\in Y}\rho_y$ by $\rho_Y$ and assume that $\rho_x$ is weakly contained in $\rho_Y$. Denoting the canonical trace on $C^*_r(\Gxx)$ by $\tau$, it follows that $\tau\circ\vartheta_{x,r}=(\rho_x(\cdot)\delta_x,\delta_x)$ lies in the weak$^*$ closure of the states $\varphi$ of the form
\begin{equation}\label{eq:phi}
\varphi=\sum^N_{i=1}(\rho_Y(\cdot)\xi_i,\xi_i),
\end{equation}
where $\xi_i$ are finitely supported functions on $s^{-1}(Y)$ such that
$$
\sum^N_{i=1}\|\xi_i\|^2=\sum^N_{i=1}\sum_{g\in s^{-1}(Y)}|\xi_i(g)|^2=1.
$$
As $\tau(\vartheta_{x,r}(f))=f(x)=1$, it suffices to consider states such that
$$
\varphi(f)>\sum^n_{k=1}\alpha_k.
$$
Since
$$
\varphi(f)=\sum^N_{i=1}\sum_{g\in s^{-1}(Y)}f(r(g))|\xi_i(g)|^2=\sum^N_{i=1}\sum_{g\in r^{-1}(Y)}f(r(g))|\xi_i(g)|^2,
$$
where we used that $r^{-1}(Y)=s^{-1}(Y)$ by the invariance of $Y$, and $f$ is zero outside $V$, this implies that
\begin{equation}\label{eq:sum-xi}
\sum^N_{i=1}\sum_{g\in r^{-1}(Y\cap V)}|\xi_i(g)|^2>\sum^n_{k=1}\alpha_k.
\end{equation}
We claim that then for some $k$ we must have $\varphi|_{A_k}\in\Omega^{(k)}_{\alpha_k,\eps}$, where $\Omega^{(k)}_{\alpha_k,\eps}$ is defined as in Lemma~\ref{lem:non-weakly-contained}.

Denote by $X_k$ the finite set of pairs $(i,g)$, $1\le i\le N$, $g\in r^{-1}(Y\cap V)$, such that $\xi_i(g)\ne0$ and $r(g)\in U_k$. As $Y\cap V\subset U_1\cup\dots \cup U_n$ by assumption, the inequality~\eqref{eq:sum-xi} implies
$$
\sum^n_{k=1}\sum_{(i,g)\in X_k}|\xi_i(g)|^2>\sum^n_{k=1}\alpha_k.
$$
It follows that for some $k$ we have
\begin{equation}\label{eq:sum-xi2}
\sum_{(i,g)\in X_k}|\xi_i(g)|^2>\alpha_k.
\end{equation}

If $(i,g)\in X_k$, then
$$
\rho_Y(f_k)\delta_g=\rho_Y(f_k^*)\delta_g=f(r(g))\delta_g.
$$
Therefore every such point $(i,g)$ defines a one-dimensional subrepresentation of $\rho_Y|_{A_k}$ and a character $\chi_{i,g}\colon A_k\to\C$ satisfying $\chi_{i,g}(f_k)>1-\eps$ by~\eqref{eq:fk}. Then on $A_k$ we have
$$
(\rho_Y(\cdot)\xi_i,\xi_i)=(\rho_Y(\cdot)\tilde\xi_i,\tilde\xi_i)+\sum_{g:(i,g)\in X_k}|\xi_i(g)|^2\chi_{i,g},
$$
where $\tilde\xi_i(g)=\xi_i(g)$ if $(i,g)\notin X_k$ and $\tilde\xi_i(g)=0$ otherwise. By~\eqref{eq:sum-xi2} this implies that $\varphi|_{A_k}\in\Omega^{(k)}_{\alpha_k,\eps}$, proving our claim.

It follows that if there is a net of states of the form~\eqref{eq:phi} that converges weakly$^*$ to $\tau\circ\vartheta_{x,r}$, then by passing to a subnet $(\varphi_j)_j$ we can find an index $k$, $1\le k\le n$, such that $\varphi_j|_{A_k}\in\Omega^{(k)}_{\alpha_k,\eps}$ for all $j$. This contradicts Lemma~\ref{lem:non-weakly-contained}.
\ep

By combining this with Proposition~\ref{prop:weak-containment} we get the following criterion.

\begin{corollary}\label{cor:criterion}
Let $\G$ be a locally compact \'etale groupoid, $Y\subset\Gu$ a $\G$-invariant subset and $x\in(\Gu\cap \bar Y)\setminus Y$. Assume that for every finite set of distinct cyclic nontrivial subgroups $G_1,\dots,G_n$ of $\Gxx$ we have
\begin{equation}\label{eq:few-subgroups}
\sum^n_{k=1}\frac{1}{|G_k|}<1.
\end{equation}
Then $\rho_x$ is weakly contained in $\bigoplus_{y\in Y}\rho_y$ if and only if Condition~\ref{condition1} is not satisfied (equivalently, if and only if there is a net in $Y$ such that $x$ is its unique accumulation point in $\Gxx$).
\end{corollary}

\bp
The ``only if'' part follows from Proposition~\ref{prop:not-weakly-contained} by observing that if Condition~\ref{condition1} is satisfied for $g_1,\dots,g_n$ and some elements $g_k$ and $g_l$ ($k\ne l$) generate the same subgroup, then Condition~\ref{condition1} is still satisfied if we omit $g_k$ or $g_l$. Therefore in Condition~\ref{condition1} we may assume in addition that the elements $g_1,\dots, g_n$ generate different subgroups.
\ep

Condition~\eqref{eq:few-subgroups} is most probably not optimal, but as the following example shows, some assumptions are needed for the conclusion of the corollary to be true.

\begin{example}
Let $X$ be the disjoint union of countably many copies of $\{0,1\}$. Consider the involutive map $S\colon X\to X$ that acts as a flip on every copy of $\{0,1\}$. Take three copies $X_1,X_2,X_3$ of $X$ and let $X^+$ be the one-point compactification of the discrete set $X_1\sqcup X_2\sqcup X_3$. Define an action of $\Gamma=(\Z/2\Z)^2$ on $X^+$ as follows: the element $(1,0)$ acts by $S$ on $X_1$ and $X_2$ and trivially on $X_3$ and $\infty$, the element $(0,1)$ acts by $S$ on $X_2$ and $X_3$ and trivially on $X_1$ and $\infty$.
Consider the corresponding groupoid of germs $\G$, so $\G$ is the quotient of the transformation groupoid $\Gamma\ltimes X^+$ by the equivalence relation defined by $(h,x)\sim(g,x)$ iff $hy=gy$ for all $y$ in a neighbourhood of~$x$. Thus, if we ignore the topology on $\G$, our groupoid is the disjoint union of $\Gamma\times\{\infty\}\cong \Gamma$ and three copies of $(\Z/2\Z)\ltimes_S X$.

Consider the set $Y=X^+\setminus\{\infty\}$ and the point $x=\infty$. Condition~\ref{condition1} is satisfied for $n=3$, since~$Y$ is discrete and for every point $y\in Y$ there is $g\in \Gamma\setminus\{0\}$ that acts trivially on $y$. We claim that nevertheless $\rho_x$ is weakly contained in $\bigoplus_{y\in Y}\rho_y$.

We have a short exact sequence
$$
0\to C^*_r(\G\setminus\Gxx)\to C^*_r(\G)\xrightarrow{\rho_x}C^*(\Gamma)\to0.
$$
It has a canonical splitting $\psi\colon C^*(\Gamma)\to C^*_r(\G)$ defined as follows. For every $g\in \Gamma$ consider the image $U_g$ of the set $\{(g,x^+)\mid x^+\in X^+\}\subset \Gamma\ltimes X^+$ in $\G$. The sets $U_g\subset\G$ are bisections and their characteristic functions span a copy of $C^*(\Gamma)$ in $C^*_r(\G)$. We define $\psi(\lambda_g)=\un_{U_g}$.

For $i=1,2,3$, let $x_{in}$ be $0$ in the $n$th copy of $\{0,1\}$ in $X_i$. Then every $f\in C_c(\G\setminus\Gxx)$ is contained in $\ker\rho_{x_{in}}$ for all $n$ sufficiently large. On the other hand, $\rho_{x_{in}}\circ\psi$ is equivalent to the representation~$\lambda_i$ obtained by composing the regular representation of $C^*(\Z/2\Z)$ with the homomorphism $C^*(\Gamma)\to C^*(\Z/2\Z)$ defined by one of the three nontrivial homomorphisms $\Gamma\to\Z/2\Z$. Namely, for $i=1$ we get the homomorphism that maps $(1,0)$ and $(1,1)$ into $1\in\Z/2\Z$, for $i=2$ it maps $(1,0)$ and $(0,1)$ into $1$, and for $i=3$ it maps $(0,1)$ and $(1,1)$ into $1$. As $\lambda_1\oplus\lambda_2\oplus\lambda_3$ is a faithful representation of $C^*(\Gamma)$ and $C_c(\G\setminus\Gxx)+\psi(C^*(\Gamma))$ is dense in $C^*_r(\G)$, it follows that $\rho_x$ is weakly contained in $\bigoplus^\infty_{n=1}\bigoplus^3_{i=1}\rho_{x_{in}}$.\er
\end{example}

We finish the section with a short discussion of essential groupoid C$^*$-algebras. Following~\cite{MR4246403}, define
$$
J_\sing=\{a\in C^*_r(\G)\mid\ \text{the set of}\ x\in \G^{(0)}\ \text{such that}\ \rho_x(a)\delta_x\ne0\ \text{is meager}\}.
$$
This is a closed ideal in $C^*_r(\G)$; in order to see that it is a right ideal, note that if $U\subset\G$ is an open bisection and $f\in C_c(U)$, then for all $x\in s(U)$ we have
\begin{equation}\label{eq:ess1}
\|\rho_x(a*f)\delta_x\|=|f(g_x)|\|\rho_{T(x)}(a)\delta_{T(x)}\|,
\end{equation}
where $g_x$ is the unique element in $U\cap\G_x$ and $T\colon s(U)\to r(U)$ is the homeomorphism defined by $T(x)=r(g_x)$. The essential groupoid C$^*$-algebra of $\G$ is defined by
$$
C^*_\ess(\G)=C^*_r(\G)/J_\sing.
$$

\begin{proposition}[cf.~{\cite{MR4246403}*{Proposition~7.18}}] \label{prop:essential}
Assume $\G$ is a locally compact \'etale groupoid that can be covered by countably many open bisections. Let $D_0\subset\Gu$ be the set of points $x\in\Gu$ satisfying the following property: there exist elements $g_1,\dots,g_n\in\Gxx\setminus\{x\}$, open bisections $U_1,\dots,U_n$ and an open neighbourhood $U$ of $x$ in $\Gu$ such that $g_k\in U_k$ for all $k$ and
$U\setminus (U_1\cup\dots \cup U_n)$ has empty interior. Let~$Y$ be a dense subset of $\Gu\setminus D_0$. Then
$$
J_\sing=\{a\in C^*_r(\G)\mid\ \text{the set of}\ x\in \G^{(0)}\ \text{such that}\ \rho_x(a)\ne0\ \text{is meager}\}=\bigcap_{y\in Y}\ker\rho_y.
$$
In particular, if $D_0=\emptyset$, then $J_\sing=0$.
\end{proposition}

Before we turn to the proof, let us make the connection to~\cite{MR4246403} more explicit. Let $D\subset\Gu$ be the set of dangerous points~\cite{MR4246403}, that is, points $x\in\Gu$ such that there is a net in $\Gu$ converging to $x$ and to an element of $\Gxx\setminus\{x\}$. It is easy to see then that $D_0\subset D$. (Should the points of $D_0$ be called extremely dangerous?)

\bp[Proof of Proposition~\ref{prop:essential}]
Let $(U_n)_n$ be a sequence of open bisections covering $\G$. For every $n$, let $T_n\colon s(U_n)\to r(U_n)$ be the homeomorphism defined by $U_n$ and $g_n\colon s(U_n)\to U_n$ be the inverse of $s\colon U_n\to s(U_n)$, so $T_n(x)=r(g_n(x))$. Similarly to~\eqref{eq:ess1}, for all $a\in C^*_r(\G)$ and $x\in s(U_n)$, we have
$$
\|\rho_x(a)\delta_{g_n(x)}\|=\|\rho_{T_n(x)}(a)\delta_{T_n(x)}\|.
$$
It follows that if $a\in J_\sing$, then the set $A_n$ of points $x\in s(U_n)$ such that $\rho_x(a)\delta_{g_n(x)}\ne0$ is meager in~$\Gu$. Then the set $\cup_nA_n$ is meager as well. Since it coincides with the set of points $x\in\Gu$ such that $\rho_x(a)\ne0$, this proves the first equality of the proposition.

For the second equality, observe first that if $x\in\Gu\setminus D$ and $\rho_x(a)\ne0$ for some $a\in C^*_r(\G)$, then $\rho_z(a)\ne0$ for all $z$ close to $x$. This follows from~\cite{MR4246403}*{Lemma~7.15} or our Proposition~\ref{prop:weak-containment}, since otherwise we could find a net $(x_i)$ converging to $x$ such that $\rho_{x_i}(a)=0$ for all $i$ and then conclude that $\rho_x(a)=0$, as $\rho_x$ is weakly contained in $\bigoplus_i\rho_{x_i}$.

The observation implies that if $a\in \cap_{y\in Y}\ker\rho_y$, then $\rho_x(a)=0$ for all $x\in\bar Y\setminus D$. The set $D$ is meager by \cite{MR4246403}*{Lemma~7.15}. As $D_0\subset D$, it follows that $Y$ is dense in~$\Gu$. Therefore if $a\in \cap_{y\in Y}\ker\rho_y$, then $\rho_x(a)$ can be nonzero only for elements $x$ of the meager set $D$, hence $a\in J_\sing$.

Conversely, assume $a\in J_\sing$. Then the observation above implies that $\rho_x(a)=0$ for all $x\in\Gu\setminus D$. Therefore to finish the proof it suffices to show that $\rho_x(a)=0$ for all $x\in D\setminus D_0$. By Proposition~\ref{prop:weak-containment}, for this, in turn, it suffices to show that for every $x\in D\setminus D_0$ Condition~\ref{condition1} is not satisfied for $Y=\Gu\setminus D$. Assume this condition is satisfied for some $x\in D$, that is, there exist elements $g_1,\dots,g_n\in\Gxx\setminus\{x\}$, open bisections $U_1,\dots,U_n$ such that $g_k\in U_k$ and a neighbourhood~$U$ of~$x$ in $\Gu$ satisfying
$U\setminus D\subset U_1\cup\dots \cup U_n$. As the set $D$ is meager, this implies that $x\in D_0$.
\ep

\bigskip

\section{\texorpdfstring{C$^*$}{C*}-algebras associated with left cancellative monoids}\label{sec:monoid}

Let $S$ be a left cancellative monoid with identity element~$e$. Consider its left regular representation
$$
\lambda\colon S\to B(\ell^2(S)),\quad \lambda_s\delta_t=\delta_{st}.
$$
The reduced C$^*$-algebra $C^*_r(S)$ of $S$ is defined as the C$^*$-algebra generated by the operators $\lambda_s$, $s\in S$.

\smallskip

Consider the left inverse hull~$I_\ell(S)$ of~$S$, that is, the inverse semigroup of partial bijections on~$S$ generated by the left translations $S\to S$. Whenever convenient we view $S$ as a subset of $I_\ell(S)$ by identifying $s$ with the left translation by $s$. For $s\in S$, we denote by $s^{-1}\in I_\ell(S)$ the bijection $sS\to S$ inverse to the bijection $S\to sS$, $t\mapsto st$. If the map with the empty domain is present in $I_\ell(S)$, we denote it by $0$.

Let $E(S)$ be the abelian semigroup of idempotents in $I_\ell(S)$. Every element of $E(S)$ is the identity map on its domain of definition $X\subset S$, which is a right ideal in $S$ of the form
$$
X=s_1^{-1}t_1\dots s_n^{-1}t_nS
$$
for some $s_1,\dots,s_n,t_1,\dots,t_n\in S$. Such right ideals are called constructible~\cite{MR2900468}. We denote by $\J(S)$ the collection of all right constructible ideals. It is a semigroup under the operation of intersection, and we have an isomorphism $E(S)\cong\J(S)$. Denote by $p_X\in E(S)$ the idempotent corresponding to $X\in\J(S)$.

Denote by $\widehat{E(S)}$ the collection of semi-characters of $E(S)$, that is, semigroup homomorphisms $E(S)\to\{0,1\}$ that are not identically zero, where $\{0,1\}$ is considered as a semigroup under multiplication. Note that every semi-character $\chi\in \widehat{E(S)}$ must satisfy $\chi(p_S)=1$.  If $0\in I_\ell(S)$, then denote by $\chi_0$ the semi-character that is identically one. This is the unique semi-character satisfying $\chi_0(0)=1$. The set $\widehat{E(S)}$ is compact Hausdorff in the topology of pointwise convergence.

Consider the Paterson groupoid $\G(I_\ell(S))$ associated with $I_\ell(S)$~\cite{MR1724106}:
$$
\G(I_\ell(S))=\Sigma/{\sim_P},\quad \text{where}\quad \Sigma = \{(g,\chi) \in I_\ell(S)\times\widehat{E(S)} \mid \chi(g^{-1}g)=1\}
$$
and the equivalence relation $\sim_P$ is defined by declaring $(g_1,\chi_1)$ and $(g_2,\chi_2)$ to be equivalent if and only if
$$
 \chi_1=\chi_2\ \ \text{and there exists}\ \ p\in E(S)\ \ \text{such that}\ \ g_1p=g_2p\ \ \text{and}\ \ \chi_1(p)=1.
$$
We denote by $[g,\chi]$ the class of $(g,\chi)\in\Sigma$ in $\G(I_\ell(S))$. The product is defined by
$$
[g,\chi]\,[h,\psi] = [gh,\psi]\quad\text{if}\quad \chi=\psi(h^{-1}\cdot h).
$$
In particular, the unit space $\G(I_\ell(S))^{(0)}$ can be identified with $\widehat{E(S)}$ via the map $\widehat{E(S)}\to\G(I_\ell(S))$, $\chi\mapsto[p_S,\chi]$, the source and range maps are given by
$$
s([g,\chi])=\chi,\qquad r([g,\chi])=\chi(g^{-1}\cdot g),
$$
while the inverse is given by $[g,\chi]^{-1}=[g^{-1},\chi(g^{-1}\cdot g)]$.

For a subset $U$ of $\widehat{E(S)}$, define
$$
D(g,U)= \{[g,\chi]\in\G(I_\ell(S)) \mid \chi \in U\}.
$$
Then the topology on $\G(I_\ell(S))$ is defined by taking as a basis the sets $D(g,U)$,
where $g\in I_\ell(S)$ and~$U$ is an open subset of the clopen set $\{\chi \in \widehat{E(S)} \mid \chi(g^{-1}g)=1\}$.
This turns $\G(I_\ell(S))$ into a locally compact, but not necessarily Hausdorff, \'{e}tale groupoid.

\smallskip

For every $s\in S$ define a semi-character~$\chi_s\in\widehat{E(S)}$~by
$$
\chi_s(p_X)=\un_X(s).
$$
The following lemma is a groupoid version of the observation of Norling~\cite{MR3200323}*{Section~3} on a connection between the regular representations of $S$ and $I_\ell(S)$. A closely related result was also proved by Spielberg~\cite{MR4151331}*{Proposition~11.4}.

\begin{lemma}\label{lem:norling}
Put $\G=\G(I_\ell(S))$ and $Z=\Gu=\widehat{E(S)}$.
Then the map $S\to \G_{\chi_e}$, $s\mapsto[s,\chi_e]$, is a bijection. If we identify $S$ with $\G_{\chi_e}$ using this map, so that the representation $\rho_{\chi_e}$ of $C^*_r(\G)$ is viewed as a representation on~$\ell^2(S)$, then
$$
\rho_{\chi_e}(C^*_r(\G))=C^*_r(S)\quad\text{and}\quad\rho_{\chi_e}(\un_{D(s,Z)})=\lambda_s\quad\text{for all}\quad s\in S.
$$
\end{lemma}

\bp
Since $\chi_e(p_J)=0$ for every constructible ideal $J$ different from $S$, we have $[s,\chi_e]=[t,\chi_e]$ only if $s=t$. This shows that the map $S\to \G_{\chi_e}$, $s\mapsto[s,\chi_e]$, is injective. In order to prove that it is surjective, assume that $(g,\chi_e)\in\Sigma$, that is, $\chi_e(g^{-1}g)=1$, for some $g\in I_\ell(S)$.
This means that the domain of definition of $g$ contains $e$, and since this domain is a right ideal, it must coincide with~$S$. But then if $s\in S$ is the image of $e$ under the action of $g$, we must have $g(t)=g(e)t=st$ for all $t\in S$, so $g=s$, which proves the surjectivity.

Next, as $\chi_e(t^{-1}\cdot t)=\chi_t$ and $[s,\chi_t]\,[t,\chi_e]=[st,\chi_e]$, we immediately get that
$\rho_{\chi_e}(\un_{D(s,Z)})=\lambda_s$ for all $s\in S$. It is not difficult to see that the C$^*$-algebra $C^*_r(\G)$ is generated by the elements~$\un_{D(s,Z)}$. (One can also refer to \cite{MR1724106}*{Theorem 4.4.2} that shows that the C$^*$-algebra $C^*_r(\G)$ is the reduced C$^*$-algebra of the inverse semigroup $I_\ell(S)$, which is generated by $S$.) Hence $\rho_{\chi_e}(C^*_r(\G))$ is exactly $C^*_r(S)$.
\ep

This lemma leads naturally to a candidate for a groupoid model for $C^*_r(S)$: define
$$
\G_P(S):=\G(I_\ell(S))_{\Omega(S)},
$$
where $\Omega(S)\subset \G(I_\ell(S))^{(0)}$ is the closure of the $\G(I_\ell(S))$-orbit of $\chi_e$. As $\chi_e(s^{-1}\cdot s)=\chi_s$, by Lemma~\ref{lem:norling} this orbit is exactly the set of semi-characters $\chi_s$, $s\in S$. The closure of this set is known and easy to find, cf.~\cite{MR3618901}*{Corollary~5.6.26}: the set $\Omega(S)=\overline{\{\chi_s\mid s\in S\}}\subset\widehat{E(S)}$ consists of the semi-characters $\chi$ satisfying the properties
\begin{enumerate}
\item[(i)] if $0\in I_\ell(S)$, then $\chi\ne\chi_0$ (equivalently, $\chi(0)=0$);
\item[(ii)] if $\chi(p_X)=1$ and $X=X_1\cup\dots\cup X_n$ for some $X,X_1,\dots,X_n\in\J(S)$, then $\chi(p_{X_i})=1$ for at least one index $i$.
\end{enumerate}
The groupoid $\G_P(S)$ is denoted by $I_l\ltimes\Omega$ in~\cite{Li-I}.

\smallskip

We are now in the setting of Section~\ref{sec:groupoid}, with $\G=\G_P(S)$ and $Y=\{\chi_s\mid s\in S\}$ a dense invariant subset of $\Gu$. Negation of Condition~\ref{condition1} leads to the following definition.

\begin{definition}\label{def:strong-regularity}
We say that $S$ is \textbf{strongly C$^*$-regular} if, given elements $h_1,\dots,h_n\in I_\ell(S)$ and constructible ideals $X,X_1,\dots,X_m\in\J(S)$ satisfying
\begin{equation}\label{eq:strong-regularity0}
\emptyset\ne X\setminus\bigcup^m_{i=1} X_i\subset \bigcup^n_{k=1}\{s\in S: h_ks=s\},
\end{equation}
there are constructible ideals $Y_1,\dots, Y_l\in \J(S)$ and indices $1\le k_j\le n$ ($j=1,\dots,l$) such that
\begin{equation}\label{eq:strong-regularity}
X\setminus\bigcup^m_{i=1} X_i\subset\bigcup^l_{j=1}Y_j\qquad\text{and}\qquad h_{k_j}p_{Y_j}=p_{Y_j}\quad\text{for all}\quad 1\le j\le l.
\end{equation}
\end{definition}

\begin{lemma}\label{lem:regularity}
Condition~\ref{condition1} is not satisfied for $\G=\G_P(S)$, $Y=\{\chi_s\mid s\in S\}$
and every $x\in\G^{(0)}\setminus Y$ if and only if $S$ is strongly C$^*$-regular.
\end{lemma}

\bp
Assume first that $S$ is strongly C$^*$-regular. Suppose there is $\chi\in \G^{(0)}\setminus Y$ such that Condition~\ref{condition1} is satisfied for $x=\chi$, and let $g_k= [h_k,\chi]$, $U_k$ ($1\le k\le n$) and $U$ be as in that condition. We may assume that $U_k=D(h_k,\Omega(S))$ and
\begin{equation}\label{eq:U-neighbourhood}
U=\{\eta\in\Omega(S)\mid \eta(p_X)=1,\ \eta(p_{X_i})=0\ \text{for}\ i=1,\dots,m\}
\end{equation}
for some $X,X_1,\dots,X_m\in\J(S)$. Then Condition~\ref{condition1} says that for every $s\in X\setminus\cup^m_{i=1} X_i$ there is~$k$ such that $\chi_s\in U_k$, that is, $h_ks=s$. By the strong C$^*$-regularity we can find $Y_1,\dots,Y_l\in\J(S)$ satisfying~\eqref{eq:strong-regularity}. As $\chi\in\Omega(S)$, there must exist $j$ such that $\chi(p_{Y_j})=1$. But then $g_{k_j}=[h_{k_j},\chi]=\chi$, which contradicts the assumption that $g_1,\dots,g_n$ are nontrivial elements of the isotropy group~$\G^\chi_\chi$.

\smallskip

Assume now that $S$ is not strongly C$^*$-regular, so there are elements $h_1,\dots,h_n\in I_\ell(S)$ and constructible ideals $X,X_1,\dots,X_m\in\J(S)$ such that \eqref{eq:strong-regularity0} holds but~\eqref{eq:strong-regularity} doesn't for any choice of~$Y_j$ and~$k_j$. In other words, if we consider the set $\F$ of all constructible ideals $J$ such that there is $k$ (depending on $J$) satisfying $h_kp_J=p_J$, then for any finite set $F\subset\F$ we have
$$
X\setminus\Big(\bigcup^m_{i=1} X_i\cup\bigcup_{J\in F}J \Big)\ne\emptyset.
$$
Pick a point $s_F$ in the above set and consider a cluster point $\chi$ of the net $(\chi_{s_F})_F$, where $F$'s are partially ordered by inclusion. Then $\chi$ lies in the set $U$ defined by~\eqref{eq:U-neighbourhood}, and $\chi(p_J)=0$ for all $J\in\F$. The semi-character $\chi$ cannot be of the form $\chi_s$, since otherwise we must have $s\in X\setminus\cup^m_{i=1} X_i$, and then $sS\in\F$ and $\chi(p_{sS})=\chi_s(p_{sS})=1$, which is a contradiction.

We claim that it is possible to replace $U$ by a smaller neighbourhood of $\chi$ and discard some of the elements $h_k$ in such a way that Condition~\ref{condition1} gets satisfied for $x=\chi$, $g_k=[h_k,\chi]$ and $U_k=D(h_k,\Omega(S))$. Namely, if $(h_k,\chi)\not\in\Sigma$ for some $k$, then we add $\dom h_k$ to the collection $\{X_1,\dots,X_m\}$ and discard such $h_k$. If $(h_k,\chi)\in\Sigma$ but $\chi(h_k^{-1}\cdot h_k)\ne\chi$, then $\chi_s(h_k^{-1}\cdot h_k)\ne\chi_s$ for all $\chi_s$ close $\chi$, so by replacing $X$ by a smaller ideal and adding more constructible ideals to $\{X_1,\dots,X_m\}$ we may assume that $\chi_s(h_k^{-1}\cdot h_k)\ne\chi_s$ for all $\chi_s\in U$ and again discard such $h_k$. For the remaining elements $h_k$ and the new $U$ we have that for every $s\in S$ such that $\chi_s\in U$ there is an index $k$ satisfying $h_ks=s$. Then, in order to show that Condition~\ref{condition1} is satisfied, it remains to check that the elements $g_k=[h_k,\chi]$ of $\G^\chi_\chi$ are nontrivial. But this is clearly true, since $\chi(p_J)=0$ for every $J\in\J(S)$ such that $h_kp_J=p_J$.
\ep

\begin{remark}\label{rem:strong-regularity}
From the last part of the proof we see that in Definition~\ref{def:strong-regularity} we may assume in addition that $X\subset\dom h_k$ for all $k$. More directly this can be seen as follows. Assume~\eqref{eq:strong-regularity0} is satisfied. Consider the nonempty subsets $F\subset\{1,\dots, n\}$ such that
$$
X_F:=X\cap\Big(\bigcap_{k\in F}\dom h_k\Big)\not\subset\bigcup^m_{i=1}X_i\cup\bigcup_{k\notin F}\dom h_k.
$$
Then~\eqref{eq:strong-regularity0} is satisfied for $X_F$, $\{X_1,\dots,X_m,\dom h_k\ (k\notin F)\}$ and $\{h_k\ (k\in F)\}$ in place of  $X$, $\{X_1,\dots,X_m\}$ and $\{h_k\ (1\le k\le n)\}$. Since
$$
X\setminus\bigcup^m_{i=1} X_i\subset\bigcup_F \Bigg(X_F\setminus\Big(\bigcup^m_{i=1}X_i\cup\bigcup_{k\notin F}\dom h_k\Big)\Bigg),
$$
we conclude that if for every $F$ condition~\eqref{eq:strong-regularity} can be satisfied for  $X_F$, $\{X_1,\dots,X_m,\dom h_k\ (k\notin F)\}$ and $\{h_k\ (k\in F)\}$, then it can be satisfied for $X$, $\{X_1,\dots,X_m\}$ and $\{h_k\ (1\le k\le n)\}$ as well.
\er
\end{remark}

Since the points $\chi_s$, $s\in S$, lie on the same $\G_P(S)$-orbit, the corresponding representations $\rho_{\chi_s}$ of $C^*_r(\G_P(S))$ are mutually equivalent. Thus, by Lemma~\ref{lem:norling} and Proposition~\ref{prop:weak-containment}, we get the following result.

\begin{proposition}\label{prop:strong-regularity}
If $S$ is a strongly C$^*$-regular left cancellative monoid, then the representation $\rho_{\chi_e}$ of $C^*_r(\G_P(S))$ defines an isomorphism
$C^*_r(\G_P(S))\cong C^*_r(S)$.
\end{proposition}

Therefore if $S$ is strongly C$^*$-regular, it is natural to define the full semigroup C$^*$-algebra of $S$ by
$$
C^*(S):=C^*(\G_P(S)).
$$

As $C^*(\G(I_\ell(S)))$ has a known description in terms of generators and relations~\cite{MR1724106}, we can quickly obtain such a description for $C^*(\G_P(S))$ as well.

\begin{proposition}[cf.~{\cite{MR4151331}*{Theorem~9.4},\cite{LS}*{Definition~3.6}}]
Assume $S$ is a countable left cancellative monoid. Consider the elements $v_s=\un_{D(s,\Omega(S))}\in C^*(\G_P(S))$, $s\in S$. Then $C^*(\G_P(S))$ is a universal unital C$^*$-algebra generated by the elements $v_s$, $s\in S$, satisfying the following relations:
\begin{enumerate}
  \item[(R1)] $v_e=1$;
  \item[(R2)] for every $g=s_1^{-1}t_1\dots s_n^{-1}t_n\in I_\ell(S)$, the element $v_g:=v_{s_1}^*v_{t_1}\dots v_{s_n}^*v_{t_n}$ is independent of the presentation of $g$;
  \item[(R3)] if $0\in I_\ell(S)$, then $v_0=0$;
  \item[(R4)] if $X=X_1\cup\dots\cup X_n$ for some $X,X_1,\dots,X_n\in\J(S)$, then
$$
\prod^n_{i=1}(v_{p_X}-v_{p_{X_i}})=0.
$$
\end{enumerate}
\end{proposition}

Note that relation (R4) is unambiguous, since relation (R2) implies that the elements $v_{p_X}$, $X\in\J(S)$, are mutually commuting projections.

\bp
Consider a universal unital C$^*$-algebra with generators $v_s$, $s\in S$, satisfying relations (R1) and (R2). This is nothing else than the full C$^*$-algebra $C^*(I_\ell(S))$ of the inverse semigroup $I_\ell(S)$, which is by definition generated by elements $v_g$, $g\in I_\ell(S)$, satisfying the relations
$$
v_gv_h=v_{gh},\qquad v_g^*=v_{g^{-1}}.
$$
By~\cite{MR1724106}*{Theorem 4.4.1}, we can identify $C^*(I_\ell(S))$ with $C^*(\G(I_\ell(S)))$. As $\G_P(S)=\G(I_\ell(S))_{\Omega(S)}$, it follows that in order to prove the proposition it remains to show that relations (R3) and (R4) describe the quotient $C^*(\G(I_\ell(S))_{\Omega(S)})$ of $C^*(\G(I_\ell(S)))$. By Proposition~\ref{prop:exact-sequence-groupoid}, the kernel of the map $C^*(\G(I_\ell(S)))\to C^*(\G(I_\ell(S))_{\Omega(S)})$ is generated as a closed ideal by the functions $$f\in C(\G(I_\ell(S))^{(0)})=C(\widehat{E(S)})$$ vanishing on $\Omega(S)$. The C$^*$-algebra $C(\widehat{E(S)})$ is a universal C$^*$-algebra generated by the projections $e_X:=\un_{U_X}$, $X\in\J(S)$, where $U_X=\{\eta\in\widehat{E(S)}:\eta(p_X)=1\}$, satisfying the relations
$e_Xe_Y=e_{X\cap Y}$. By the definition of $\Omega(S)$, relations (R3) and (R4), with $e_X$ instead of $v_{p_X}$, describe the quotient $C(\Omega(S))$ of $C(\widehat{E(S)})$. This gives the result.
\ep

\begin{remark}
The assumption of countability of $S$ is certainly not needed in the above proposition, we added it to be able to formally apply results of~\cite{MR1724106}.
\end{remark}

\begin{remark}
Relation (R2) can be slightly relaxed, cf.~\cite{LS}*{Definition~3.6}: it suffices to require that the elements $v_g$ are well-defined only for $g=p_X$, $X\in\J(S)$. Indeed, then, given $g=a_1^{-1}b_1\dots a_n^{-1}b_n=c_1^{-1}d_1\dots c_m^{-1}d_m\in I_\ell(S)$, for the elements $v=v_{a_1}^*v_{b_1}\dots v_{a_n}^*v_{b_n}$ and $w=v_{c_1}^*v_{d_1}\dots v_{c_m}^*v_{d_m}$ we have $v^*v=v^*w=w^*v=w^*w$, hence $(v-w)^*(v-w)=0$ and $v=w$.
\er
\end{remark}

Let us next give a few sufficient conditions for strong C$^*$-regularity. Recall that a left cancellative monoid $S$ is called finitely aligned~\cite{MR4151331}, or right (ideal) Howson~\cite{ES}, if the right ideal $sS\cap tS$ is finitely generated for all $s,t\in S$. If $S$ is finitely aligned, then by induction on $n$ one can see that the constructible ideals $s_1^{-1}t_1\dots s_n^{-1}t_nS$ are finitely generated.

\begin{proposition}\label{prop:strong-regularity-suff}
A left cancellative monoid $S$ is strongly C$^*$-regular if either of the following conditions is satisfied:
\begin{enumerate}
  \item the groupoid $\G_P(S)$ is Hausdorff;
  \item the monoid $S$ is group embeddable;
  \item the monoid $S$ is finitely aligned.
\end{enumerate}
\end{proposition}

We remark that apart from the easy implication $(2)\Rightarrow(1)$, which will be explained shortly, there are no relations between conditions (1)--(3). For example, certain Baumslag--Solitar monoids are group embeddable but are not finitely aligned~\cite{MR3000828}*{Lemma~2.12}. Examples of finitely aligned monoids with non-Hausdorff $\G_P(S)$ can be found among the Zappa--Sz\'ep products $G\bowtie X^*$ defined by self-similar actions of groups on free monoids with infinite number of generators~\cite{MR2395198}; see Remark~\ref{rem:non-Haus} for a related in spirit example with trivial group of units. Nevertheless, if $S$ is finitely aligned and right cancellative, then $\G_P(S)$ is Hausdorff, see~\cite{MR4151331}*{Lemma~7.1} or~\cite{Li-I}*{Remark~4.3}.

\bp[Proof of  Proposition~\ref{prop:strong-regularity-suff}]
Condition (1) is obviously sufficient for strong C$^*$-regularity, since Condition~\ref{condition1} can be satisfied only for non-Hausdorff groupoids.

Condition (2) is known, and is easily seen, to be stronger than (1): if $S$ is a submonoid of a group~$G$, then every nonzero element $g$ of $I_\ell(S)$ acts by the left translation by an element $h_g\in G$, and we have either $h_g=e$ and $D(g,\Omega(S))\subset\G_P(S)^{(0)}$ or $h_g\ne e$ and $D(g,\Omega(S))\cap\G_P(S)^{(0)}=\emptyset$.

Assume now that $S$ is finitely aligned and~\eqref{eq:strong-regularity0} is satisfied. Choose a finite set of generators of the right ideal $X$. Let $s_1,\dots,s_l$ be those generators that do not lie in $X_1\cup\dots\cup X_m$. By assumption, for every $1\le j\le l$ we can find $1\le k_j\le n$ such that $h_{k_j}s_j=s_j$. Then~\eqref{eq:strong-regularity} is satisfied for $Y_j=s_jS$.
\ep

\begin{remark}
By~\cite{Li-I}*{Lemma~4.1}, the groupoid $\G_P(S)$ is Hausdorff if and only if whenever $g\in I_\ell(S)$ and $\{s\in S\mid gs=s\}\ne\emptyset$, there are $Y_1,\dots,Y_l\in\J(S)$ such that $\{s\in S\mid gs=s\}=Y_1\cup\dots\cup Y_l$. Using this characterization one can easily see that (1) implies strong C$^*$-regularity without relying on Lemma~\ref{lem:regularity}.\er
\end{remark}

In addition to $\G_P(S)$ there is another closely related groupoid associated to $S$, which was introduced by Spielberg~\cite{MR4151331} and which we will now turn to. %We will essentially follow the presentation by Li~\cite{Li-I}.

Consider the collection $\bar\J(S)$ of subsets of $S$ obtained by adding to $\J(S)$ all sets of the form $X\setminus\cup^m_{i=1}X_i$ for $X,X_1,\dots,X_m\in\J(S)$. It is a semigroup under intersection. When we want to view $E(S)$ as its subsemigroup, we will write $p_X$ instead of $X$ for the elements of $\bar\J(S)$. Every $\chi\in\Omega(S)$ extends to a semi-character on $\bar\J(S)$ by letting $\chi(p_{X\setminus\cup^m_{i=1}X_i})=1$ ($X,X_1,\dots,X_m\in\J(S)$) if $\chi(p_X)=1$ and $\chi(p_{X_i})=0$ for all $i$, and $\chi(p_{X\setminus\cup^m_{i=1}X_i})=0$ otherwise.

Now, define an equivalence relation $\sim$ on $\G_P(S)$ by declaring $[g,\chi]\sim[h,\chi]$ iff there exists $X\in\bar\J(S)$ such that $\chi(p_X)=1$ and $g|_X=h|_X$. Consider the quotient groupoid
$$
\G(S):=\G_P(S)/\sim.
$$
This groupoid is denoted by $G_2(S)$ in~\cite{MR4151331} and by $I_l\bar\ltimes\Omega$ in~\cite{Li-I}.

Similarly to Lemma~\ref{lem:norling}, the representation $\rho_{\chi_e}$ of $C^*_r(\G(S))$ defines a surjective $*$-homomorphism $C^*_r(\G(S))\to C^*_r(S)$. Negation of Condition~\ref{condition1} for $\G=\G(S)$, $Y=\{\chi_s\mid s\in S\}$ and all $x\in\G^{(0)}\setminus Y$ leads to the following definition.

\begin{definition}\label{def:regularity}
We say that $S$ is \textbf{C$^*$-regular} if, given $h_1,\dots,h_n\in I_\ell(S)$ and $X\in\bar\J(S)$ satisfying
\begin{equation*}\label{eq:regularity0}
\emptyset\ne X\subset \bigcup^n_{k=1}\{s\in S: h_ks=s\},
\end{equation*}
there are sets $Y_1,\dots, Y_l\in\bar\J(S)$ and indices $1\le k_j\le n$ ($j=1,\dots,l$) such that
\begin{equation*}\label{eq:regularity}
X\subset\bigcup^l_{j=1}Y_j\qquad\text{and}\qquad h_{k_j}|_{Y_j}=\id\quad\text{for all}\quad 1\le j\le l.
\end{equation*}
\end{definition}
Note that by the same argument as in Remark~\ref{rem:strong-regularity}, in order to check C$^*$-regularity it suffices to consider $X=X_0\setminus(X_1\cup\dots\cup X_m)$ ($X_i\in\J(S)$) and $h_k$ such that $X_0\subset \dom h_k$ for all $k$. Note also that the only difference between C$^*$-regularity and strong C$^*$-regularity is that the sets~$Y_j$ are required to be in $\bar\J(S)$ in the first case and in $\J(S)$ in the second. In particular, strong C$^*$-regularity implies C$^*$-regularity.

\smallskip

Similarly to Proposition~\ref{prop:strong-regularity} we get the following result.

\begin{proposition}\label{prop:regularity}
If $S$ is a C$^*$-regular left cancellative monoid, then the representation $\rho_{\chi_e}$ of $C^*_r(\G(S))$ defines an isomorphism
$C^*_r(\G(S))\cong C^*_r(S)$.
\end{proposition}

Thus, if $S$ is C$^*$-regular, we can define, following~\cite{MR4151331}, the full semigroup C$^*$-algebra of $S$ by
$$
C^*(S):=C^*(\G(S)).
$$
A presentation of $C^*(\G(S))$ in terms of generators and relations is given in~\cite{MR4151331}*{Theorem~9.4}.

We therefore have two candidates for a groupoid model of $C^*_r(S)$, and hence two potentially different definitions of full semigroup C$^*$-algebras associated with $S$. As the following result shows, it is $\G(S)$ which is the preferred model and we have only one candidate for $C^*(S)$.

\begin{proposition}\label{prop:G_P=G}
Assume $S$ is a left cancellative monoid such that $\rho_{\chi_e}$ defines an isomorphism $C^*_r(\G_P(S))\cong C^*_r(S)$. Then $\G_P(S)=\G(S)$.
\end{proposition}

\bp
Assume $\G_P(S)\ne\G(S)$. Then there exists $[g,\chi]\in \G_P(S)$ such that $[g,\chi]\ne\chi$ but $[g,\chi]\sim\chi$. Let $X,X_1,\dots,X_m\in\J(S)$ be such that $\chi(p_X)=1$, $\chi(p_{X_i})=0$ for all $i$ and $gs=s$ for all $s\in X\setminus\cup^m_{i=1}X_i$. Consider the clopen set $U\subset\G_P(S)^{(0)}$ defined by~\eqref{eq:U-neighbourhood} and the function $f=\un_{D(g,U)}-\un_U$ on $\G_P(S)$. Then $f\ne0$, but if $g=s_1^{-1}t_1\dots s_n^{-1}t_n$ and we identify $\ell^2(\G_P(S)_{\chi_e})$ with $\ell^2(S)$, then
$$
\rho_{\chi_e}(f)=(\lambda_{s_1}^*\lambda_{t_1}\dots \lambda_{s_n}^*\lambda_{t_n}-1)\un_{X\setminus\cup^m_{i=1}X_i}=0.
$$
Therefore $\rho_{\chi_e}\colon C^*_r(\G_P(S))\to C^*_r(S)$ has a nontrivial kernel.
\ep

\begin{corollary}
If $S$ is strongly C$^*$-regular, then $S$ is C$^*$-regular and $\G_P(S)=\G(S)$.
\end{corollary}

\bp
The first statement follows from the definitions, as was already observed after Definition~\ref{def:regularity}. The equality $\G_P(S)=\G(S)$ follows from Propositions~\ref{prop:strong-regularity} and~\ref{prop:G_P=G}.
\ep

In particular, by Proposition~\ref{prop:strong-regularity-suff}, if $\G_P(S)$ is Hausdorff or $S$ is finitely aligned, then $\G_P(S)=\G(S)$. This has been already known, see~\cite{Li-I}*{Lemma~3.2}.

The equality $\G_P(S)=\G(S)$ in the strong C$^*$-regular case is also an immediate consequence of the following criterion.

\begin{lemma}\label{lem:g=gp}
For every left cancellative monoid $S$, we have $\G_P(S)=\G(S)$ if and only if $S$ has the following property: given $g\in I_\ell(S)$ and  $X,X_1,\dots,X_m\in\J(S)$ such that $gs=s$ for all $s\in X\setminus\cup^m_{i=1}X_i\ne\emptyset$, there are $Y_1,\dots,Y_l\in\J(S)$ such that
$$
X\setminus\bigcup^m_{i=1} X_i\subset\bigcup^l_{j=1}Y_j\qquad\text{and}\qquad gp_{Y_j}=p_{Y_j}\quad\text{for all}\quad 1\le j\le l.
$$
\end{lemma}

\bp
Assume first that the condition in the formulation of the lemma is satisfied. In order to prove that $\G_P(S)=\G(S)$, we have to show that if $[g,\chi]\sim\chi$ for some $g\in I_\ell(S)$ and $\chi\in\Omega(S)$, then $[g,\chi]=\chi$. Let $X,X_1,\dots,X_m\in\J(S)$ be such that $\chi(p_X)=1$, $\chi(p_{X_i})=0$ for all $i$ and $gs=s$ for all $s\in X\setminus\cup^m_{i=1}X_i$. By assumption, there are $Y_1,\dots,Y_l\in\J(S)$ such that $X\setminus\bigcup^m_{i=1} X_i\subset\bigcup^l_{j=1}Y_j$ and $gp_{Y_j}=p_{Y_j}$ for all $1\le j\le l$. But then $\chi(p_{Y_j})=1$ for some $j$, hence $[g,\chi]=\chi$.

Assume now that the condition in the formulation is not satisfied. Then, similarly to the proof of Lemma~\ref{lem:regularity}, we can find $g\in I_\ell(S)$,  $X,X_1,\dots,X_m\in\J(S)$ and $\chi\in\Omega(S)$ such that $\chi(p_X)=1$, $\chi(p_{X_i})=0$ for all $i$, $\chi(p_J)=0$ for all $J\in\J(S)$ satisfying $gp_J=p_J$, and $gs=s$  for all $s\in X\setminus\cup^m_{i=1}X_i$. Then $[g,\chi]\sim\chi$ and $[g,\chi]\ne\chi$, so $\G_P(S)\ne\G(S)$.
\ep

We finish this section by showing that under rather general assumptions $C^*_r(S)$ coincides with the essential groupoid C$^*$-algebras of $\G_P(S)$ and $\G(S)$.

\begin{proposition}\label{prop:essential1}
Let $S$ be a countable left cancellative monoid. Assume that for any nontrivial units $u_1,\dots,u_k\in S^*\setminus\{e\}$ and every $X\in\bar\J(S)$ containing $e$, there is $Z\in\bar\J(S)$ such that $\emptyset\ne Z\subset X$ and $u_is\ne s$ for all $s\in Z$ and $i=1,\dots,k$.
Then the maps $\rho_{\chi_e}\colon C^*_r(\G_P(S))\to C^*_r(S)$ and $\rho_{\chi_e}\colon C^*_r(\G(S))\to C^*_r(S)$ define isomorphisms
$$
C^*_\ess(\G_P(S))\cong C^*_r(S)\cong C^*_\ess(\G(S)).
$$
\end{proposition}

\bp
Consider the groupoid $\G_P(S)$.  Let $Y=\{\chi_s|s\in S\}$ and $D_0$ be the set defined in Proposition~\ref{prop:essential}. As $Y$ is dense in $\Omega(S)$ and $\ker\rho_{\chi_s}$ is independent of $s$, by Proposition~\ref{prop:essential} in order to prove the first isomorphism it suffices to show that $D_0\cap Y=\emptyset$.

Since both sets $D_0$ and $Y$ are invariant, it is enough to show that $\chi_e\notin D_0$. By Lemma~\ref{lem:norling}, the isotropy group $\G_P(S)^{\chi_e}_{\chi_e}$ consists of the elements $[u,\chi_e]$, $u\in S^*$. Therefore we need to show that if $u_1,\dots,u_k\in S^*\setminus\{e\}$ and $U$ is a neighbourhood of $\chi_e$ in $\Omega(S)$, then the set $U\setminus\cup^k_{i=1}D(u_i,\Omega(S))$ has nonempty interior. By the definition of the topology on $\Omega(S)$, we can find $X\in\bar\J(S)$ such $e\in X$ and $\{\chi\mid\chi(p_X)=1\}\subset  U$. Let $Z\in\bar\J(S)$ be as in the formulation of the proposition. Then the clopen set $V=\{\chi\mid\chi(p_Z)=1\}$ is contained in $U$. We claim that it does not intersect $D(u_i,\Omega(S))$ for all~$i$.

Assume $\chi\in V\cap D(u_i,\Omega(S))$ for some $i$. This means that $[u_i,\chi]=[e,\chi]$, that is, there is $W\in\J(S)$ such that $\chi(p_W)=1$ and $u_is=s$ for all $s\in W$. But then $\chi(p_{Z\cap W})=1$, so $Z\cap W$ is nonempty, contradicting the property $u_is\ne s$ for all $s\in Z$.

This proves the proposition for $\G_P(S)$, the proof for $\G(S)$ is essentially the same.
\ep

\bigskip

\section{Example of a regular monoid}

Consider the monoid $S$ given by the monoid presentation
\begin{equation}\label{eq:ex1}
S=\langle a, b, x_n, y_n \ (n\in\mathbb{Z}) : abx_n=bx_n,\  aby_n=by_{n+1} \ (n\in \mathbb{Z})\rangle.
\end{equation}

Our goal is to prove the following.

\begin{proposition}\label{prop:monoid-S}
The monoid $S$ defined by~\eqref{eq:ex1} is left cancellative and strongly C$^*$-regular. It is not finitely aligned and the groupoid $\G_P(S)=\G(S)$ is not Hausdorff.
\end{proposition}

The proof is divided into several lemmas. But first we need to introduce some notation. Consider the set $\mathbb S$ of finite words (including the empty word) in the alphabet $\{a,b,x_n,y_n\ (n\in\Z)\}$. We say that two words are equivalent if they represent the same element of $S$. Let $\tau\subset\mathbb S\times\mathbb S$ be the symmetric set of relations defining~$S$,~so
$$
\tau=\{(abx_n,bx_n),(bx_n,abx_n),(aby_n,by_{n+1}),(by_{n+1},aby_n)\ (n\in\Z)\}.
$$
By a \emph{$\tau$-sequence} we mean a finite sequence $s_0,\dots,s_n$ of words such that for every $i=1,\dots, n$ we can write $z_{i-1}=c_ip_id_i$ and $z_i=c_iq_id_i$ with $(p_i,q_i)\in\tau$. Then by definition two words $s$ and $t$ are equivalent if and only if there is a $\tau$-sequence $s_0,\dots,s_n$ with $s_0=s$ and $s_n=t$.

\begin{lemma}\label{lem:S-cancellation}
The monoid $S$ is left cancellative.
\end{lemma}

\bp
It will be convenient to use the following notation. For words $s$ and $t$, let us write $s\perp_0 t$ if every word $abx_n$, $bx_n$, $aby_n$, $by_n$ ($n\in\Z$) in $st$ that begins in $s$ ends in $s$. We write $s\perp t$ if for all words $s'$ and $t'$ such that $s\sim s'$ and $t\sim t'$, we have $s'\perp_0 t'$.
We will repeatedly use that if $s\perp t$ and $st\sim w$ for a word $w$, then $w=s't'$ for some $s'\sim s$ and $t'\sim t$.

In order to prove the lemma, it suffices to show that for all letters $x$ and words $w$, $w'$, the equivalence $xw\sim xw'$ implies that $w \sim w'$.

\smallskip
\underline{Case $x=x_n, y_n$:}\\
The only word equivalent to $x$ is $x$ itself, and we have $x \perp w$, so the equivalence $xw \sim xw'$ furnishes the equivalence $w \sim w'$.

\smallskip
\underline{Case $x=a$:}\\
Write $xw=a^kv$ ($k\geq 1$), with $v$ not starting with $a$. Consider several subcases.

Assume $v$ is empty or starts with $x_n$ or $y_n$. Then $a^k\perp v$. The only word equivalent to $a^k$ is $a^k$ itself. It follows that $xw'=a^kv'$, with $v'\sim v$, and therefore $w=a^{k-1}v\sim a^{k-1}v'=w'$.

Assume now that $v$ starts with $b$ and write $xw=a^kbu$. Assume $u$ is empty or starts with~$a$ or~$b$. Then $a^kb\perp u$. The only word equivalent to $a^kb$ is $a^kb$ itself. It follows that $xw'=a^kbu'$, with $u'\sim  u$, hence $w\sim w'$.

Assume next that $u$ starts with $x_n$ and write $xw=a^kbx_nz$. Then $a^kbx_n\perp z$. The words equivalent to $a^kbx_n$ are $a^mbx_n$ ($m\ge0$).
It follows that $xw'=a^mbx_nz'$ for some $m\ge1$ and $z'\sim z$, hence $w\sim w'$.

Similarly, if $u$ starts with $y_n$, write $xw=a^kby_nz$. Then $a^kby_n\perp z$. The words equivalent to $a^kby_n$ are $a^mby_{n+m-k}$ ($m\ge0$). It follows that $xw'=a^mby_{n+m-k}z'$ for some $m\ge1$ and $z'\sim z$, hence $w\sim w'$.

\smallskip
\underline{Case $x=b$:}\\
If $w$ is empty or starts with $a$ or $b$, then $x \perp w$ and we are done similarly to the first case above.

If $w$ starts with $x_n$, write $xw=bx_nv$. Then $bx_n\perp v$. The words equivalent to $bx_n$ are $a^mbx_n$ ($m\geq 0$), and the only one among them beginning with $b$ is $bx_n$ itself. It follows that $xw'=bx_nv'$, with $v'\sim v$, hence $w\sim w'$. The case when $w$ starts with $y_n$ is similar, since the only word beginning with $b$ that is equivalent to $by_n$ is $by_n$ itself.
\ep

Next we want to describe the constructible ideals of $S$. We start with the following lemma.

\begin{lemma}\label{lem:S-principal-intersections}
We have:
\begin{enumerate}
  \item if $x\in \{x_n, y_n\ (n\in \Z)\}$, $y\in\{a,b,x_n,y_n\ (n\in\Z)\}$ and $x\ne y$, then
$$
xS\cap yS=\emptyset;
$$
\item for all $k\geq 1$,
$$
bS\cap a^kS=bS\cap a^kbS=\bigcup_{n\in\Z} bx_nS\cup \bigcup_{n\in\Z} by_nS.
$$
\end{enumerate}
\end{lemma}

\begin{proof}
(1) None of the words occurring in the defining relations of $S$ begins with $x_n$ or $y_n$. From this it follows that a word beginning with $x_n$, resp., $y_n$, can only be equivalent to a word beginning with $x_n$, resp., $y_n$. Thus, $xS\cap yS=\emptyset$.

\smallskip

(2) Since we have $bx_n=a^kbx_n$ and $by_n=a^kby_{n-k}$ for all $n$ and $k$, it is clear that
$$
\bigcup_n bx_nS\cup \bigcup_n by_nS \subset bS\cap a^kbS\subset bS\cap a^kS.
$$

To prove the opposite inclusions, assume $s\in bS\cap a^kS$. Take words $w$ and $w'$ in $\mathbb S$ such that $s$ is represented by $bw$ and $a^kw'$. There is a $\tau$-sequence $z_0,\dots,z_m$ such that $bw=z_0$ and $a^kw'=z_m$. We have $z_{i-1}=c_ip_id_i$ and $z_i = c_i q_i d_i$, with $(p_i, q_i)\in \tau$. There must be an index $i$ such that $c_i=\emptyset$ and $p_i$ starts with $b$. But then $p_i=bx_n$ or $p_i=by_n$ for some $n$, since these are the only words in the defining relations of $S$ that begin with $b$.
Therefore $s$ lies in $bx_nS$ or in $by_nS$. \end{proof}

This lemma already implies that $S$ is not finitely aligned, since $b^{-1}aS=\bigcup_n x_nS\cup \bigcup_n y_nS$ by~(2) and
the sets $x_nS$, $y_mS$ are disjoint for all $n$ and $m$ by (1), so the right ideal $b^{-1}aS$ is not finitely generated.

\begin{lemma}\label{lem:constructible-S}
The constructible ideals of $S$ are
\begin{equation}\label{eq:S-constructible-ideals}
\emptyset,\quad sS,\quad \bigcup_{n\in\Z} sx_nS\cup \bigcup_{n\in\Z} sy_nS\quad (s\in S).
\end{equation}
\end{lemma}

\bp
By Lemma~\ref{lem:S-principal-intersections}, the ideals in~\eqref{eq:S-constructible-ideals} are constructible. In order to prove the lemma it is then enough to show that for every $x\in\{a,b,x_n,y_n\ (n\in\Z)\}$ and $s\in S$, the right ideals $x^{-1}sS$ and $\bigcup_n x^{-1}sx_nS\cup \bigcup_n x^{-1}sy_nS$ are again of the form~\eqref{eq:S-constructible-ideals}. This is obviously true when $s\in xS$. By Lemma~\ref{lem:S-principal-intersections}(1) this is also true if $s=e$. So from now on we assume that $s\notin xS$ and $s\ne e$.

\smallskip
\underline{Case $x=x_n, y_n$:}\\
In this case, from Lemma~\ref{lem:S-principal-intersections}(1) we see that the sets $x^{-1}sS$ and $\bigcup_n x^{-1}sx_nS\cup \bigcup_n x^{-1}sy_nS$ are empty.

\smallskip
\underline{Case $x=a$:}\\
Again, Lemma~\ref{lem:S-principal-intersections}(1) tells us that the sets $a^{-1}sS$ and $\bigcup_n a^{-1}sx_nS\cup \bigcup_n a^{-1}sy_nS$ are empty if $s\in x_mS$ or $s\in y_mS$ for some $m$. As $s\notin aS$ and $s\ne e$, we may therefore assume that $s\in bS$. Consider several subcases.

Assume $s=b$. Then, by Lemma~\ref{lem:S-principal-intersections}(2), $aS\cap bS=\bigcup_n bx_nS\cup \bigcup_n by_nS$. As $a$ maps this set onto itself, we conclude that both
$a^{-1}bS$ and $\bigcup_n a^{-1}bx_nS\cup \bigcup_n a^{-1}by_nS$ are equal to $\bigcup_n bx_nS\cup \bigcup_n by_nS$.

Next, assume $s\in baS$ or $s\in b^2S$. From the defining relations we see that every word in $\mathbb S$ that starts with $ba$ or $b^2$ can only be equivalent to a word that again starts with $ba$ or $b^2$. Hence the sets $a^{-1}baS$ and $a^{-1}b^2S$ are empty, and therefore $a^{-1}sS$ and $\bigcup_n a^{-1}sx_nS\cup \bigcup_n a^{-1}sy_nS$ are empty as well.

It remains to consider the subcase when $s\in bx_mS$ or $s\in by_mS$. But then $s\in aS$, which contradicts our assumption on $s$.

\smallskip
\underline{Case $x=b$:}\\
Similarly to the previous case, we may assume that $s\in aS$. Write $s=a^kt$ for some $k\geq 1$ and $t\in S$ such that $t$ can be represented by a word not starting with $a$.
Consider several subcases.

Assume $t=e$. Then, using Lemma~\ref{lem:S-principal-intersections}(2), we get
$$
b^{-1}a^kS=b^{-1}(bS\cap a^kS)=b^{-1}\Big(\bigcup_n bx_nS\cup \bigcup_n by_nS\Big)=\bigcup_n x_nS\cup \bigcup_n y_nS.
$$
Every word in $\mathbb S$ that starts with $a^kx_n$  can only be equivalent to a word that again starts with $a^kx_n$. The same is true for $y_n$ in place of $x_n$. Hence
\begin{equation}\label{eq:constructible-S}
\bigcup_n b^{-1}a^kx_nS\cup \bigcup_n b^{-1}a^ky_nS=\emptyset.
\end{equation}

Next, assume $t\in x_mS$ or $t\in y_mS$. Then~\eqref{eq:constructible-S} implies that both $b^{-1}a^ktS$ and $\bigcup_n b^{-1}a^ktx_nS\cup \bigcup_n b^{-1}a^kty_nS$ are empty.

It remains to consider the subcase $t\in bS$. This splits into several subsubcases.

If $t=b$, then using Lemma~\ref{lem:S-principal-intersections}(2) again,
$$
b^{-1}a^kbS=b^{-1}\Big(bS\cap a^kbS\Big)=b^{-1}\Big(\bigcup_n bx_nS\cup \bigcup_n by_nS\Big)=\bigcup_n x_nS\cup \bigcup_n y_nS.
$$
As $a$ maps $\bigcup_n bx_nS\cup \bigcup_n by_nS$ onto itself, we also have
$$
\bigcup_n b^{-1}a^kbx_nS\cup \bigcup_n b^{-1}a^kby_nS=\bigcup_n x_nS\cup \bigcup_n y_nS.
$$

Assume next that $t\in baS$ or $t\in b^2S$. Every word in $\mathbb S$ that starts with $a^kba$ or $a^kb^2$  can only be equivalent to a word that again starts with $a^kba$ or $a^kb^2$ . Hence the sets $b^{-1}a^kbaS$ and $b^{-1}a^kb^2S$ are empty, and therefore $b^{-1}a^ktS$ and $\bigcup_n b^{-1}a^ktx_nS\cup \bigcup_n b^{-1}a^kty_nS$ are empty as well.

Finally, assume $t\in bx_mS$ or $t\in by_mS$. But then $s=a^kt\in bS$, which contradicts our assumption on $s$.
\ep

We now look at the topology on $\G_P(S)$. We will need the following lemma.

\begin{lemma}\label{lem:bxn}
The constructible ideals that contain at least two elements $bx_n$ are
\begin{equation}\label{eq:bxn}
S,\quad a^kbS\ \ (k\ge0),\quad \bigcup_{n\in\Z} bx_nS\cup \bigcup_{n\in\Z} by_nS.
\end{equation}
\end{lemma}

\bp
It is clear that the ideals in~\eqref{eq:bxn} contain $bx_n$ for all $n\in\Z$. Assume $X\in\J(S)$ contains $bx_l$ and $bx_m$ for some $l\ne m$. Observe that the words in $\mathbb S$ equivalent to $bx_l$ for a fixed $l$ are $a^kbx_l$, $k\ge0$. It follows that if $bx_l\in sS$ for some $s$, then any word in $\mathbb S$ representing $s$ has the form $a^k$, $a^kb$ or $a^kbx_l$ for some $k\ge0$. Consider two cases.

Assume $X=sS$. By the above observation, the only possibilities for $s$ to have $bx_l,bx_m\in sS$ are $s=a^k$ or $s=a^kb$ for some $k\ge0$, so $X$ has the required form.

Assume $X=\bigcup_n sx_nS\cup \bigcup_n sy_nS$. If $bx_l\in sx_nS$ for some $n$, then by the above observation $n=l$ and $s=a^kb$ for some $k\ge0$, so $X$ has the required form. Otherwise we must have $bx_l\in sy_nS$ for some $n$, but this is not possible, again by the above observation.
\ep

This lemma implies that the semi-characters $\chi_{bx_n}$ on $E(S)$ converge as $n\to\pm\infty$ to a semi-character $\chi$ such that $\chi(p_X)=1$ for all $X$ in~\eqref{eq:bxn} and $\chi(p_X)=0$ for all other constructible ideals.

\begin{lemma}\label{lem:S-non-Hausdorff}
The semi-characters $\chi_{bx_n}$ converge in $\G_P(S)$ to the different elements $\chi$ and $[a,\chi]$, so $\G_P(S)$ is not Hausdorff.
\end{lemma}

\bp
The convergence $\chi_{bx_n}\to[a,\chi]$ follows from the fact that $[a, \chi_{bx_n}]=[e, \chi_{bx_n}]$ for all $n$, because~$a$ fixes $bx_n$. The element $[a,\chi]\in\G_P(S)^\chi_\chi$ is nontrivial, since by Lemma~\ref{lem:bxn} if $X\in \J(S)$ and $\chi(p_X)=1$, then $X$ contains the elements $by_n$ that $a$ does not fix.
\ep

It remains to show that $S$ is strongly C$^*$-regular. Recall that by Definition~\ref{def:strong-regularity} and Remark~\ref{rem:strong-regularity} this means that, given elements $h_1,\dots,h_N\in I_\ell(S)$ and constructible ideals $X,X_1,\dots,X_M\in\J(S)$ satisfying
\begin{equation*}
X\subset\bigcap^N_{k=1}\dom h_k,\qquad \emptyset\ne X\setminus\bigcup^M_{i=1} X_i\subset \bigcup^N_{k=1}\{s\in S: h_ks=s\},
\end{equation*}
we need to show that there are constructible ideals $Y_1,\dots, Y_L\in \J(S)$ and indices $1\le k_j\le N$ ($j=1,\dots,L$) such that
\begin{equation*}
X\setminus\bigcup^M_{i=1} X_i\subset\bigcup^L_{j=1}Y_j\qquad\text{and}\qquad h_{k_j}p_{Y_j}=p_{Y_j}\quad\text{for all}\quad 1\le j\le L.
\end{equation*}
We will actually show more: there is an index $k$ such that $h_kp_X=p_X$.

This is obviously true when $X$ is a principal right ideal, cf.~Proposition~\ref{prop:strong-regularity-suff}(3). Therefore we need only to consider $X=\bigcup_{n} sx_nS\cup \bigcup_{n} sy_nS$. By replacing $X$, $X_i$ and $h_k$ by $s^{-1}X$, $s^{-1}X_i$ and $s^{-1}h_ks$ we may assume that
$$
X=\bigcup_{n\in\Z} x_nS\cup \bigcup_{n\in\Z} y_nS.
$$

\begin{lemma}\label{lem:S-technical-1}
The only constructible ideals that contain $y_n$ for a fixed $n$ are $S$, $y_nS$ and $X$.
\end{lemma}

\bp
A principal ideal $sS$ contains $y_n$ only if $s=e$ or $s=y_n$, since the only word in $\mathbb S$ that is equivalent to $y_n$ is $y_n$ itself. For the same reason if an ideal
$\bigcup_k sx_kS\cup \bigcup_k sy_kS$ contains $y_n$, then we must have $s=e$, so the ideal is $X$.
\ep

Since by assumption $X\setminus\cup^M_{i=1}X_i\ne\emptyset$, it follows that every ideal $X_i$ contains at most one element~$y_n$. Therefore $X\setminus\cup^M_{i=1}X_i$ contains $y_n$ for all but finitely many $n$'s. In particular, there are indices $m$ and $k$ such that $h_ky_m=y_m$. To finish the proof of strong C$^*$-regularity it is now enough to establish the following.

\begin{lemma}\label{lem:fix-y-m}
If $h\in I_\ell(S)$ fixes $y_m$ for some $m$ and satisfies $\dom h\supset X=\bigcup_n x_nS\cup \bigcup_n y_nS$, then $h=p_S=e$ or $h=p_X$.
\end{lemma}

\begin{proof}
For $k\ge1$, consider the element $g_k=a^{-k}b\in I_\ell(S)$. Note that by Lemma~\ref{lem:S-principal-intersections}(2) we have $\dom g_k=X$ and $g_kX=bX$. For $k\in\Z\setminus\{0\}$, consider $g_k'=b^{-1}a^kb\in I_\ell(S)$. The domain and range of $g_k'$ is $X$.

We will show that if $h\in I_\ell(S)$ satisfies $\dom h\supset X$, then $h$
must be of the form
\begin{equation}\label{eq:hX}
sg_kp_J,\quad sg_k'p_J\quad \text{or}\quad sp_J\quad (s\in S,\ J\in\J(S)).
\end{equation}
This will yield the proposition. Indeed, first of all, we then have $J=S$ or $J=X$, since by Lemma~\ref{lem:S-technical-1} these are the only constructible ideals containing $X$. Then $hy_m=sby_{m+k}$, $hy_m=sy_{m+k}$ or $hy_m=sy_m$, resp., and this equals $y_m$ only in the third case with $s=e$.

\smallskip

Take $h\in I_\ell(X)$ that is not of the form~\eqref{eq:hX}. We will show that $\dom h\not\supset X$. Begin by writing $h$ as a word in the generators of $S$ and their inverses. We may assume that there are no occurrences of $x^{-1}x$ or $xx^{-1}$ in this word for each generator $x$. Indeed, $x^{-1}x=e$ can be omitted, while $xx^{-1}=p_{xS}$ and if $h=h_1 p_{xS}h_2$, then $h=h_1 h_2 p_{h_2^{-1}xS}$, so instead of $h$ we can consider $h_1h_2$. We may also assume that there are no occurrences of $a^kbx_n$ and $a^kby_n$ for $k\in \mathbb{Z}\setminus\{0\}$ and $n \in \mathbb{Z}$, as these can be simplified to $bx_n$ and $by_{n+k}$, resp. We then say that $h$ (or, more precisely, our word for $h$) is reduced.

The assumption that $h$ is not of the form~\eqref{eq:hX} implies that our word for $h$ has the form $h_1x^{-1}wg$, where $h_1$ is a word in the generators and their inverses, $x$ is one of the generators, $w$ is a word in the generators and we have one of the following options for $g$: (i) $g=g_k'$; (ii) $g=g_k$ and either $w$ is not trivial or $x\ne a, b$; (iii) $g=e$ and neither ($x=a$ and $w=b$) nor ($x=b$ and $w=a^mb$ for some $m\ge1$). Without loss of generality we may assume that $h_1$ is trivial. We are not going to distinguish between the word $w\in\mathbb S$ and the element of $S$ it represents.

\smallskip
\underline{Case $x=x_n,y_n$:}\\
By the proof of Lemma~\ref{lem:constructible-S}, for every word $v$ in the generators, we have $x^{-1}vS=\emptyset$ unless $v$ is empty or starts with $x$. Since~$w$ does not start with~$x$ and we have $g_kX=bX$, it follows that $x^{-1}wgX=\emptyset$  unless $w$ is empty and either $g=g_k'$ or $g=e$. In order to deal with the remaining cases we have to show that $g_k'X=X\not\subset xS$, which is clearly true by Lemma~\ref{lem:S-principal-intersections}(1).

\smallskip
\underline{Case $x=a$:}\\
By the proof of Lemma~\ref{lem:constructible-S}, for every word $v$ in the generators, we have $a^{-1}vS=\emptyset$ unless $v$ is empty, $v=b$ or $v$ starts with $a$, $bx_n$ or $by_n$. Since $w$ cannot start with $a$, $bx_n$ or $by_n$, it follows that $a^{-1}wgX=\emptyset$ unless we have one of the following: (i) $w$ is empty and $g=g_k$; (ii) $w=b$ and $g=g_k'$; (iii) $w=b$ and $g=e$. Cases (i) and (iii) are not possible by our assumptions on $x$, $w$ and $g$. Case (ii) is not possible either, as the word $bg_k'=bb^{-1}a^kb$ is not reduced.

\smallskip
\underline{Case $x=b$:}\\
By the proof of Lemma~\ref{lem:constructible-S}, for every word $v$ in the generators, we have $b^{-1}vS=\emptyset$ unless $v$ is empty, $v=a^m$, $v=a^mb$ or $v$ starts with $b$, $a^mbx_n$ or $a^mby_n$ ($m\ge1$, $n\in\Z$). Since $w$ cannot start with $b$, $a^mbx_n$ or $a^mby_n$, it follows that $b^{-1}wgX=\emptyset$ unless we have one of the following: (i) $w$ is empty and $g=g_k$; (ii) $w=a^m$ and $g=g_k$; (iii) $w=a^mb$ and $g=g_k'$; (iv) $w=a^mb$ and $g=e$. Cases (i) and (iv) are not possible by our assumptions on $x$, $w$ and $g$. Cases (ii) and (iii) are not possible either, as the words $a^mg_k=a^ma^{-k}b$ and $a^mbg_k'=a^mbb^{-1}a^kb$ are not reduced.
\end{proof}

This finishes the proof of Proposition~\ref{prop:monoid-S}.

\begin{remark}\label{rem:non-Haus}
If we drop the generator $b$ and consider
$$
\tilde S=\langle a, x_n, y_n \ (n\in\mathbb{Z}) : ax_n=x_n,\  ay_n=y_{n+1} \ (n\in \mathbb{Z})\rangle,
$$
then we get a left cancellative right LCM monoid, meaning that every constructible ideal is either empty or principal. Similarly to Lemma~\ref{lem:S-non-Hausdorff},
the semi-characters $\chi_{x_n}$ converge to different elements~$\chi$ and $[a,\chi]$, so the groupoid $\G_P(\tilde S)=\G(\tilde S)$ is not Hausdorff.
\end{remark}

\bigskip

\section{Example of a nonregular monoid}

Consider a small modification $T$ of the monoid $S$ from the previous section:
\begin{align}
T=\langle a, b, c, x_n, y_n \ (n\in\mathbb{Z}) : & \ \ abx_n=bx_n,\ aby_n=by_{n+1},\nonumber\\
 &\ \ cbx_n=bx_{n+1},\  cby_n=by_n\ (n\in \mathbb{Z})\rangle.\label{eq:monoid-T}
\end{align}

For this monoid we have the following result.

\begin{proposition}\label{prop:monoid-T}
The monoid $T$ defined by~\eqref{eq:monoid-T} is left cancellative. It is not C$^*$-regular, furthermore, the homomorphism $\rho_{\chi_e}\colon C^*_r(\G(T))\to C^*_r(T)$ has nontrivial kernel.
\end{proposition}

A large part of the analysis of $T$ is similar to that of $S$, so we will omit most of it.

\smallskip

The claim that $T$ is left cancellative is proved similarly to Lemma~\ref{lem:S-cancellation}, the main difference being that powers of $a$ get replaced by products of $a$ and $c$.
The description of the constructible ideals is exactly the same as for $S$ (Lemma~\ref{lem:constructible-S}):

\begin{lemma}\label{lem:constructible-T}
The constructible ideals of $T$ are
\begin{equation*}\label{eq:T-constructible-ideals}
\emptyset,\quad tT,\quad \bigcup_{n\in\Z} tx_nT\cup \bigcup_{n\in\Z} ty_nT\quad (t\in T).
\end{equation*}
\end{lemma}

The next result is similar to Lemma~\ref{lem:bxn}.

\begin{lemma}\label{lem:bxn-T}
The constructible ideals of $T$ that contain at least two elements $bx_n$ or at least two elements $by_n$ are
\begin{equation}\label{eq:bxn-T}
T,\quad tbT\ \ (t\ \text{is a product of}\ a,c),\quad \bigcup_{n\in\Z} bx_nT\cup \bigcup_{n\in\Z} by_nT.
\end{equation}
\end{lemma}

This lemma implies that the semi-characters $\chi_{bx_n}$ converge as $n\to\pm\infty$ to a semi-character $\chi$ such that $\chi(p_X)=1$ for all $X$ in~\eqref{eq:bxn-T} and $\chi(p_X)=0$ for all other $X\in\J(T)$. The semi-characters $\chi_{by_n}$ converge to $\chi$ as well. The following lemma finishes the proof of Proposition~\ref{prop:monoid-T}.

\begin{lemma}
The representation $\rho_\chi$ of $C^*_r(\G(T))$ is not weakly contained in $\rho_{\chi_e}$, hence
$$
\rho_{\chi_e}\colon C^*_r(\G(T))\to C^*_r(T)
$$
has nontrivial kernel.
\end{lemma}

\bp
Consider the constructible ideal $X=\bigcup_n bx_nT\cup \bigcup_n by_nT$, the neighbourhood $U=\{\eta\in\Omega(T):\eta(p_X)=1\}$ of~$\chi$ and the elements $g_1=[a,\chi]$ and $g_2=[c,\chi]$ of $\G(T)^\chi_\chi$. That $g_1$ and $g_2$ are indeed nontrivial elements of the isotropy group is proved as in Lemma~\ref{lem:S-non-Hausdorff}; furthermore, we can see that these are elements of infinite order generating a copy of $\Z^2$. Now, if $\chi_t\in U$ for some $t\in T$, then $t\in X$ and hence $t$ is fixed by~$a$ or by~$c$. It follows that either $\chi_t\in D(a,U)$ or $\chi_t\in D(c,U)$. We thus see that $x=\chi$ satisfies Condition~\ref{condition1} for $Y=\{\chi_t: t\in T\}$. By Proposition~\ref{prop:not-weakly-contained} we conclude that $\rho_\chi$ is not weakly contained in $\rho_{\chi_e}$.
\ep

At this point it is actually not difficult to exhibit an explicit nonzero element of $\ker\rho_{\chi_e}$: consider, for example, the function
$$
(\un_{D(a,\Omega(T))}-\un_{\Omega(T)})*(\un_{D(c,\Omega(T))}-\un_{\Omega(T)})*\un_U\in C_c(\G(T)).
$$
Its restriction to $\G(T)^\chi_\chi$ is nonzero, since $g_1$ and $g_2$ are nontrivial elements of $\G(T)^\chi_\chi$. It lies in the kernel of $\rho_{\chi_e}$, since
$$
(\lambda_a-1)(\lambda_c-1)\un_X=0\quad\text{on}\quad\ell^2(T).
$$

\smallskip

Therefore $\G(T)$ does not provide a groupoid model for $C^*_r(T)$. Since the unit group of $T$ is trivial, by Proposition~\ref{prop:essential1} we nevertheless have
$$
C^*_r(T)\cong C^*_\ess(\G(T)).
$$

\begin{remark}
An argument similar to the proof of Lemma~\ref{lem:fix-y-m} shows that if $h\in I_\ell(T)$ fixes~$x_k$ and~$y_m$ for some $k$ and $m$ and satisfies $\dom h\supset X=\bigcup_n x_nT\cup \bigcup_n y_nT$, then $h=p_T$ or $h=p_X$. Using this it is not difficult to show that $\G_P(T)=\G(T)$.
\end{remark}

\bigskip

\end{document}